\documentclass[reqno]{amsart}

\usepackage{amssymb, upref, titlesec, epic}

 \titleformat{\section}[block]{\normalfont \bfseries\large}{\hfil\thesection.}{0.25em}{}
 \titleformat{\subsection}[runin]{\normalfont\bf}{\thesubsection.}{0.25em}{}[.\quad]

\renewcommand{\see}[1]{(see~\ref{#1})}

\newcommand{\cf}[1]{(cf.~\ref{#1})}

\newcommand{\iso}{\cong}
\newcommand{\out}[3]{{}^{#1} \ifx #2A\mkern-5mu \else\mkern-4mu\fi{#2}_{#3}}
\renewcommand{\int}{\mathop{\rm int}}

\newcommand{\integer}{{\mathord{\mathbb{Z}}}}

\newcommand{\rational}{{\mathord{\mathbb{Q}}}}
\newcommand{\real}{{\mathord{\mathbb{R}}}}
\newcommand{\complex}{{\mathord{\mathbb{C}}}}
\newcommand{\Gmult}{{\mathord{\mathbb{G}_m}}}

\newcommand{\aG}{\mathbf{G}}
\newcommand{\aH}{\mathbf{H}}
\newcommand{\aM}{\mathbf{M}}
\newcommand{\aT}{\mathbf{T}}
\newcommand{\aS}{\mathbf{S}}
\newcommand{\aR}{\mathbf{R}}
\newcommand{\aB}{\mathbf{B}}
\newcommand{\aC}{\mathbf{C}}
\newcommand{\aZ}{\mathbf{Z}}
\newcommand{\aP}{\mathbf{P}}
\DeclareMathOperator{\aSpin}{\mathop{\mathbf{Spin}}}

\newcommand{\fieldrank}[1]{\mathop{\mbox{\upshape$#1$-rank}}}
\newcommand{\Qrank}{\fieldrank{\rational}}
\newcommand{\Rrank}{\fieldrank{\real}}
\newcommand{\Frank}{\fieldrank{F}}
\newcommand{\locrank}{\fieldrank{F_v}}
\DeclareMathOperator{\rank}{rank}
\newcommand{\ints}{\mathcal{O}}
\newcommand{\SG}{S_{\aG}}

\DeclareMathOperator{\SL}{SL}
\DeclareMathOperator{\SU}{SU}

\DeclareMathOperator{\SO}{SO}

\DeclareMathOperator{\Sp}{Sp}
\newcommand{\Mat}[1]{M_{#1}}
\DeclareMathOperator{\Gal}{Gal}

\DeclareMathOperator{\aSL}{\mathbf{SL}}
\DeclareMathOperator{\aSO}{\mathbf{SO}}
\DeclareMathOperator{\aSU}{\mathbf{SU}}

\DeclareMathOperator{\aSp}{\mathbf{Sp}}

\newcommand{\res}[1]{\mathop{\mathbf{R}_{#1}}}

\newcommand{\pref}[1]{{\upshape(}\ref{#1}{\upshape)}}
\newcommand{\fullref}[2]{{\ref{#1}\pref{#1-#2}}}

\renewcommand{\thesubsection}{\arabic{section}\Alph{subsection}}
\numberwithin{equation}{section}

\numberwithin{figure}{section}

\newtheoremstyle{myplain}{}{}{\normalfont\itshape}{}{\normalfont\bf}{.}{ }%
{\thmnumber{{\normalfont(#2)}}\thmname{ #1}\thmnote{ \normalfont\textup{#3}}}

\newtheoremstyle{mydefinition}{}{}{\normalfont}{}{\normalfont\bf}{.}{ }%
{\thmnumber{{\normalfont(#2)}}\thmname{ #1}\thmnote{ \normalfont\textup{#3}}}

\newtheoremstyle{myremark}{}{}{\normalfont}{}{\normalfont\itshape}{.}{ }%
{\ignorespaces\thmnumber{{\normalfont(#2)}}\thmname{ #1}\thmnote{ \normalfont\textup{#3}}}

\newtheoremstyle{myremark*}{}{}{\normalfont}{}{\normalfont\itshape}{.}{ }%
{\ignorespaces\thmname{ #1}\thmnote{ \normalfont\textup{#3}}}

\theoremstyle{myplain}
\newtheorem{thm}[equation]{Theorem}
\newtheorem{lem}[equation]{Lemma}
\newtheorem{cor}[equation]{Corollary}
\newtheorem{prop}[equation]{Proposition}

\theoremstyle{mydefinition}
\newtheorem{defn}[equation]{Definition}
\newtheorem{assump}[equation]{Assumption}
\newtheorem{assumps}[equation]{Assumptions}
\newtheorem{eg}[equation]{Example}

\newtheorem{notation}[equation]{Notation}

\theoremstyle{myremark}
\newtheorem{rem}[equation]{Remark}

\newtheorem{fact}[equation]{Fact}

\theoremstyle{myremark*}
\newtheorem{acks}{Acknowledgments}

 \newenvironment{claim}[1][\unskip]{\em
 \medskip \noindent Claim.\ }{\unskip\upshape}

 \newcounter{step}
 \newenvironment{step}[1][\unskip]{\refstepcounter{step}\em
 \medskip \noindent Step \thestep\ #1.\ }{\unskip\upshape}
 \renewcommand{\thestep}{\arabic{step}}

 \newcounter{case}
 \newenvironment{case}[1][\unskip]{\refstepcounter{case}\it
 \medskip \noindent Case \thecase\ #1.\it\ }{\unskip\upshape}
 \renewcommand{\thecase}{\arabic{case}}

 \newcounter{subcase}
 \newenvironment{subcase}[1][\unskip]{\refstepcounter{subcase}\em
 \medskip \noindent Subcase \thesubcase\ #1.\ }{\unskip\upshape}
 \numberwithin{subcase}{case}

 \newcounter{subsubcase}
 
 \numberwithin{subsubcase}{subcase}

\begin{document}

\title[Almost-minimal lattices of higher rank]
 {Almost-minimal nonuniform lattices of higher rank}

\author[V.\,Chernousov]{Vladimir Chernousov}
 \address{Department  of Mathematical and Statistical Sciences,
 University of Alberta,
 Edmonton,  Alberta, T6G~2G1, Canada}
 \email{chernous@math.ualberta.ca}

\author[L.\,Lifschitz]{Lucy Lifschitz}
 \address{Department of Mathematics,
 University of Oklahoma,
 Norman, Oklahoma, 73019, USA}
 \email{LLifschitz@math.ou.edu, 
 http://www.math.ou.edu/$\sim$llifschitz/}

\author[D.\,W.\,Morris]{Dave Witte Morris}
 \address{Department of Mathematics and Computer Science,
 University of Lethbridge,
 Lethbridge, Alberta, T1K~3M4, Canada}
 \email{Dave.Morris@uleth.ca, 
 http://people.uleth.ca/$\sim$dave.morris/}
 
 \thanks{V.\,Chernousov was partially supported by the Canada Research
Chairs Program and an NSERC research grant.
D.\,W.\,Morris was partially supported by research grants from NSERC and the NSF}

\date{November 6, 2007} 

\maketitle

\section{Introduction}

We find the minimal elements in three different (but essentially equivalent) partially ordered categories of
mathematical objects:
\begin{enumerate} \renewcommand{\theenumi}{\Alph{enumi}}
\item \label{category-geom}
 finite-volume, noncompact, complete, locally symmetric spaces of higher rank,
\item \label{category-Lie}
 nonuniform, irreducible lattices in semisimple Lie groups of higher real rank,
and
\item \label{category-Qgrp}
 isotropic, simple algebraic $\rational$-groups of higher real rank.
\end{enumerate}
The main interest is in categories \pref{category-geom} and~\pref{category-Lie}, but the proof is carried out using the machinery of~\pref{category-Qgrp}.
(For completeness, we also provide a generalization that applies to algebraic groups over any number field, not only~$\rational$.)
Justification of the examples and facts stated in the introduction can be found in \S\ref{justify}.

\subsection{Locally symmetric spaces} \label{SymmspaceSubsec}

It is well known that if $G$ is a connected, semisimple Lie group, and $\Rrank G \ge 2$, then $G$ contains a closed subgroup that is locally isomorphic to either $\SL_3(\real)$ or $\SL_2(\real) \times \SL_2(\real)$. Passing from semisimple Lie groups to the corresponding symmetric spaces yields the following geometric translation of this observation.

\begin{fact} \label{GlobalFact}
Let $\widetilde{X}$ be a symmetric space of noncompact type, with no Euclidean factors, such that $\rank \widetilde{X} \ge 2$. Then $\widetilde{X}$ contains a totally geodesic submanifold~$\widetilde{X}'$, such that $\widetilde{X}'$ is the symmetric space associated to either $\SL_3(\real)$ or $\SL_2(\real) \times \SL_2(\real)$.
In other words, $\widetilde{X}'$ is isometric to either
\begin{enumerate} \renewcommand{\itemsep}{\medskipamount}
\item \label{GlobalFact-SL3}
$\SL_3(\real)/\SO(3) \iso 
\left\{ \begin{matrix}
\mbox{$3 \times 3$ positive-definite symmetric} \\
\mbox{real matrices of determinant~$1$} 
\end{matrix}
\right\}$,
or
\item the product $\mathbb{H}^2 \times \mathbb{H}^2$ of $2$ hyperbolic planes.
\end{enumerate}
\end{fact}

In short, among all the symmetric spaces of noncompact type with rank $\ge 2$, there are only two manifolds that are minimal with respect to the partial order defined by totally geodesic embeddings. Our main theorem provides an analogue of this result for noncompact finite-volume spaces that are locally symmetric, rather than globally symmetric, but, in this setting, the partial order has infinitely many minimal elements.

\begin{thm} \label{LocSymmThm}
Let $X$ be a finite-volume, noncompact, irreducible, complete, locally symmetric space of noncompact type, with no Euclidean factors {\rm(}locally\/{\rm)}, such that $\rank X \ge 2$. Then there is a finite-volume, noncompact, irreducible, complete, locally symmetric space~$X'$, such that $X'$ admits a totally geodesic, proper immersion into~$X$, and the universal cover of~$X'$ is the symmetric space associated to either\/
$\SL_3(\real)$, $\SL_3(\complex)$, or a direct product\/ $\SL_2(\real)^m \times \SL_2(\complex)^n$, with $m + n \ge 2$.
\end{thm}

\begin{rem}
The symmetric space associated to $\SL_3(\real)$ is given in~\fullref{GlobalFact}{SL3}. The others are:
\begin{enumerate} \renewcommand{\itemsep}{\medskipamount}
\item $\SL_3(\complex)/\SU(3) \iso
\left\{ \begin{matrix}
\mbox{$3 \times 3$ positive-definite Hermitian} \\
\mbox{matrices of determinant~$1$} 
\end{matrix}
\right\}$,
and
\item the product $(\mathbb{H}^2)^m \times (\mathbb{H}^3)^n$ of $m$~hyperbolic planes and $n$~hyperbolic $3$-spaces. 
\end{enumerate}
\end{rem}

\begin{rem} \ 
\begin{enumerate}
\item Our main result actually provides a precise description of~$X'$ (modulo finite covers), not only its universal cover. It does this by specifying the fundamental group $\pi_1(X')$; the possible fundamental groups appear in \S\ref{LatticeSubsec}.
\item Our proof of the theorem is constructive: for a given locally symmetric space~$X$, our methods produce an explicit locally symmetric space~$X'$ that embeds in~$X$.
\item Our theorem assumes $X$ is \emph{not} compact. It would be interesting to obtain an analogous result that assumes $X$ is compact (and $X'$ is also compact). 
\end{enumerate}
\end{rem}

The Mostow Rigidity Theorem tells us that any locally symmetric space~$X$ as discussed above is determined by its fundamental group. This means that the above geometric result can be reformulated in group-theoretic terms. This restatement of the result is our next topic.

\subsection{Lattices in semisimple Lie groups} 
\label{LatticeSubsec}

\begin{defn}
Let us say that an abstract group~$\Gamma$ is a \emph{nonuniform lattice
of higher rank} if
 there exists a connected, semisimple, linear (real) Lie group~$G$,
 such that
 \begin{itemize}
 \item $\Gamma$ is isomorphic to an irreducible, \emph{nonuniform}
lattice in~$G$,
 and
 \item $\Rrank G \ge 2$.
 \end{itemize}
 (Recall that a discrete subgroup~$\Gamma'$ of~$G$ is a \emph{nonuniform
lattice} if $G/\Gamma'$ has finite volume, but is \emph{not} compact. The
lattice~$\Gamma'$ is \emph{irreducible} if no finite-index subgroup
of~$\Gamma'$ is isomorphic to a direct product $\Gamma_1' \times
\Gamma_2'$ with both $\Gamma_1'$ and~$\Gamma_2'$ infinite.)
\end{defn}

\begin{rem}
The nonuniform lattices of higher rank have made many appearances in the literature.
For example, the Margulis Arithmeticity Theorem \cite{[M73]} was first proved for this class of groups, and M.\,S.\,Raghunathan \cite{[R76],[R86]} proved the Congruence Subgroup Property for these groups.
\end{rem}

It is obvious that the collection of all nonuniform lattices of higher rank is
closed under passage to finite-index subgroups, so 
it has no elements that are minimal under inclusion. Thus, it is
natural to consider a slightly weaker notion of minimality that
ignores finite-index subgroups.

\begin{defn}
 A nonuniform lattice~$\Gamma$ of higher rank is \emph{almost minimal} if no
subgroup of infinite index in~$\Gamma$ is a nonuniform lattice of higher rank.
 \end{defn}
 
Our main result describes all the almost-minimal nonuniform lattices of higher rank.
The significance of this result lies in the fact that every nonuniform lattice of
higher rank must contain an almost-minimal one, so, for example, they can be the base
cases in a proof by induction.

\begin{thm}
Every almost-minimal nonuniform lattice of higher rank is isomorphic to a nonuniform, irreducible lattice in either\/ $\SL_3(\real)$, $\SL_3(\complex)$, or a direct product\/ $\SL_2(\real)^m \times \SL_2(\complex)^n$, with $m + n \ge 2$.
\end{thm}

We now describe the almost-minimal lattices more explicitly.

\begin{eg} \label{Qrank2eg}
 $\SL_3(\integer)$ is an almost-minimal nonuniform lattice of higher rank. 
\end{eg}

\begin{rem}
$\SL_3(\integer)$ is an arithmetic group whose $\Qrank$
is~$2$. It is well known that any irreducible lattice~$\Gamma$ with $\Qrank \Gamma
\ge 2$ must contain a finite-index subgroup of either $\SL_3(\integer)$ or 
$\Sp_4(\integer)$, and one can show that $\Sp_4(\integer)$ is \emph{not}
almost minimal. Therefore, up to finite-index, $\SL_3(\integer)$ is the only
almost-minimal lattice of higher rank whose $\Qrank$ is $\ge 2$.
 \end{rem}

Although (up to finite index) there is only one almost-minimal nonuniform lattice whose
$\Qrank$ is~$2$, there are infinitely many whose $\Qrank$ is~$1$.

\begin{eg} \label{SL2Egs} \ 
 \begin{enumerate} \itemsep=\smallskipamount
 \item \label{SL2Egs-quad}
 If $r$ is any square-free integer $\ge 2$, then $\SL_2 \bigl(
\integer \bigl[\sqrt{r} \bigr] \bigr)$ is an almost-minimal nonuniform lattice of higher rank.
 \item \label{SL2Egs-any}
 More generally, let $\ints_K$ be the ring of integers of an algebraic
number field~$K$, and assume $K$ is neither $\rational$ nor an imaginary
quadratic extension of~$\rational$. Then $\Gamma = \SL_2(\ints_K)$ is a
nonuniform lattice of higher rank. (We remark that if $K$ is a totally real extension
of~$\rational$, as is the case in~\pref{SL2Egs-quad}, then
$\Gamma$ is called a \emph{Hilbert modular group}.) This nonuniform lattice is almost minimal
if and only if each proper subfield of~$K$ is either~$\rational$ or an
imaginary quadratic extension of~$\rational$.
\end{enumerate}
\end{eg}

\begin{eg} \label{SU2Egs} \ 
 Let
 \begin{itemize}
 \item $F$ be either the field~$\rational$ or an imaginary quadratic
extension of~$\rational$, 
 \item $F_v = 
 \begin{cases}
 \real & \text{if $F = \rational$} , \\
 \complex & \text{if $F \not\subset \real$} ,
 \end{cases}$
 \item $L$ be any quadratic extension of~$F$, such that $L \subset
F_v$,
 \item $\tau$ be the nontrivial Galois automorphism of~$L$ over~$F$,
 \item $f$ be the $\tau$-Hermitian form on~$L^3$ defined by
 $$ f(x,y) = \tau(x_1) \, y_1 - \tau(x_2) \, y_2 - \tau(x_3) \, y_3 ,$$
 and
 \item $\ints_L$ be the ring of integers of~$L$.
 \end{itemize}
 Then 
 $$ \SU_3(\ints_L, f, \tau) = \{\, A \in \SL_3(\ints_L) \mid f(Ax,Ay) =
f(x,y), \ \forall x,y \in L^3 \,\} $$
 is a nonuniform lattice in $\SL_3(F_v)$, so it is a nonuniform lattice of higher rank. 
  It is almost minimal if and only if either $F = \rational$ or $L \cap
\real = \rational$.
 \end{eg}

The preceding examples are well known (and are of classical type). Our main result shows there are no others:

\begin{thm} \label{LattThm}
 Every nonuniform lattice of higher rank contains a subgroup that is isomorphic to a finite-index
subgroup of a lattice described in Example\/ \ref{Qrank2eg}, \fullref{SL2Egs}{any}, or\/~\ref{SU2Egs}.
 \end{thm}

\begin{rem}
Theorem~\ref{LattThm} is a fundamental ingredient in the proof \cite{LifschitzMorris-RO} that if all nonuniform lattices of higher rank are boundedly generated by unipotent elements, then no nonuniform lattice of higher rank can be right ordered.
\end{rem}

The Margulis Arithmeticity Theorem tells us that (modulo finite groups) any nonuniform lattice of higher rank can be realized as the integral points of a simple algebraic $\rational$-group. Also, the Margulis Superrigidity Theorem tells us that any embedding $\Gamma' \hookrightarrow \Gamma$ extends to an embedding of the corresponding algebraic $\rational$-groups (modulo finite groups). This means that the classification of almost-minimal nonuniform lattices of higher rank is logically equivalent to a result on simple algebraic $\rational$-groups.

\subsection{Simple algebraic $\rational$-groups} \label{QgroupSubsec}

Let $\aG$ be a connected algebraic group over~$\rational$ that is almost simple. (Recall that, by definition, this means every proper, normal $\rational$-subgroup of~$\aG$ is finite.) It is well known that if $\Qrank \aG \ge 2$, then $\aG$ contains a $\rational$-split almost simple subgroup~$\aH$, such that $\Qrank \aH = 2$. (Indeed, one may choose $\aH$ to be isogenous to either $\aSL_3$ or $\aSp_4$.) If we replace the assumption that $\aG$ has large $\Qrank$ with the weaker assumption that $\aG$ has large $\Rrank$, then one cannot expect to find a subgroup of large $\Rrank$ that is split over~$\rational$. (In any $\rational$-split subgroup, the $\rational$-rank and $\Rrank$ are equal.) However, our main result states that if we add the obvious necessary condition that $\aG$ is $\rational$-isotropic, then there is always a subgroup of large $\Rrank$ that is \emph{quasi}split over~$\rational$:

\begin{thm} \label{MainQuasisplit}
 Suppose\/ $\aG$ is an isotropic, almost simple algebraic group over\/~$\rational$, such that\/ $\Rrank \aG \ge 2$. Then\/ $\aG$ has a connected, isotropic, almost simple\/
$\rational$-subgroup\/~$\aH$, such that\/ $\aH$ is quasisplit over\/~$\rational$, and\/ $\Rrank \aH \ge 2$.
 \end{thm}
 
It was mentioned above that $\aH$ can be chosen to be isogenous to either $\aSL_3$ or $\aSp_4$ if $\Qrank \aG \ge 2$. So the theorem is only interesting when $\Qrank \aG = 1$. Because there are very few quasisplit groups of $\Qrank$~$1$ (and it is not difficult to find quasisplit proper subgroups of $\aSp_4$, as will be seen in the proof of Lemma~\ref{FrankMust1}), we can restate the result in the following more precise form. 
 
\begin{defn}
Suppose $\aG$ is an isotropic, almost simple algebraic group over\/~$\rational$, such that\/ 
$\Rrank \aG \ge 2$.
For convenience, let us say that $\aG$ is \emph{minimal} if no proper, isotropic, almost simple\/ $\rational$-subgroup of~$\aG$ has real rank $\ge 2$.
\end{defn}

\begin{notation}
$\res{K/F}$ denotes the Weil restriction of scalars functor from~$K$ to~$F$.
\end{notation}

 \begin{thm} \label{QgrpExplicitThm}
 Suppose\/ $\aG$ is an isotropic, almost simple algebraic group over\/~$\rational$, such that\/ 
$\Rrank \aG \ge 2$. If\/ $\aG$ is minimal, then\/ $\aG$ is isogenous to either:
\begin{enumerate} \renewcommand{\theenumi}{\roman{enumi}}
\item \label{QgrpExplicitThm-SL3}
$\aSL_3$,
or
\item \label{QgrpExplicitThm-SUreal}
 $\aSU_3(L,f,\tau)$, where $L$ is a real quadratic extension of\/~$\rational$, $\tau$~is the Galois automorphism of~$L$ over\/~$\rational$, and 
	\begin{equation} \label{QgrpExplicitThm-f}
	f(x_1,x_2,x_3) = \tau(x_1) \, x_1 - \tau(x_2) \, x_2 - \tau(x_3) \, x_3 
	, \end{equation}
or
\item \label{QgrpExplicitThm-SUcplx}
 $\res{K/\rational} \aSU_3(L,f,\tau)$, where $K$ is an imaginary quadratic extension of\/~$\rational$, $L$~is a quadratic extension of~$K$, $\tau$~is the Galois automorphism of~$L$ over~$K$, and $f$~is given by \pref{QgrpExplicitThm-f},
or
\item \label{QgrpExplicitThm-SL2}
 $\res{K/\rational} \aSL_2$, for some finite extension~$K$ of\/~$\rational$, such that $K$ is neither\/~$\rational$, nor an imaginary quadratic extension of\/~$\rational$.
\end{enumerate}
 \end{thm}
 
 \begin{rem} \label{QgrpsAreMin}
 Conversely:
 \begin{enumerate} 
 \renewcommand{\theenumi}{\roman{enumi}}
 \renewcommand{\itemsep}{\smallskipamount}
 \item $\aSL_3$ is minimal.
 \item The groups described in \fullref{QgrpExplicitThm}{SUreal} are minimal.
 \item A group as described in \fullref{QgrpExplicitThm}{SUcplx} fails to be minimal if and only if $L$ contains a real quadratic extension of~$\rational$.
 \item A group as described in \fullref{QgrpExplicitThm}{SL2} fails to be minimal if and only if $K$ contains a proper subfield that is neither~$\rational$ nor an imaginary quadratic extension of~$\rational$.
 \end{enumerate}
 \end{rem}

\begin{rem}
Under the additional assumption that some minimal parabolic $\real$-subgroup of~$\aG$ is defined over~$\rational$, Theorem~\ref{QgrpExplicitThm} was proved long ago by G.\,A.\,Margulis \cite[Lem.~2.4.2]{[M75]} and M.\,S.\,Raghunathan \cite[Lem.~3.2(ii)]{[R86]} (independently). 
\end{rem}

\begin{rem}
Theorem~\ref{MainForF} provides a generalization of Theorem~\ref{QgrpExplicitThm} that applies to algebraic groups over any algebraic number field.
\end{rem}

\subsection*{Outline of the paper}
Section~\ref{justify} justifies statements made in the above introduction. 
The remaining sections of the paper state and prove our main result (Theorem~\ref{MainForF}).
Section~\ref{PrelimSect} covers some preliminaries, ands deals with groups that either have global rank $\ge 2$ or are of type $E_7$, $E_8$, or~$G_2$.
We treat
groups of classical type in \S\ref{ClassicalSect},
groups of type~$F_4$ in \S\ref{F4Sect},
groups of type  $\out{3,6}D4$ in~\S\ref{TrialitySect},
and
groups of type  $\out{1,2}E6$ in~\S\ref{E6Sect}.

\section{Justification of the introduction} \label{justify}

In this section, we provide brief justifications for the assertions made in the introduction. The order of the topics there was chosen for purposes of exposition; they will now be treated in reverse order (\S\ref{QgroupSubsec}, \S\ref{LatticeSubsec}, \S\ref{SymmspaceSubsec}).

\subsection*{Justification of \S\ref{QgroupSubsec}}
The observation that $\rational$-simple groups of higher $\rational$-rank contain subgroups isogenous to either $\aSL_3$ or~$\aSp_4$ appears in \cite[Prop.~I.1.6.2, p.~46]{MargulisBook}.

The following sections will present a proof of (a generalization of) Theorem~\ref{QgrpExplicitThm}. Because all of the groups in the conclusion of \pref{QgrpExplicitThm} are quasisplit, Theorem~\ref{MainQuasisplit} is an immediate consequence.

To verify the observations in Remark~\ref{QgrpsAreMin}, note that:
\begin{itemize} \renewcommand{\itemsep}{\smallskipamount}

\item The groups described in \fullref{QgrpExplicitThm}{SL3} and~\fullref{QgrpExplicitThm}{SUreal} are isomorphic to $\aSL_3$ over the algebraic closure~$\overline{\rational}$. Since $\aSL_3$ has no semisimple, proper subgroups of absolute rank $\ge 2$, it is immediate that these groups are minimal.

\item Let $\aG$ be one of the groups described in  \fullref{QgrpExplicitThm}{SUcplx}. 
If $L$ contains a real quadratic extension~$F$ of~$\rational$, then $\aG$ contains $\aSU_3(F,f,\tau|_F)$, so $\aG$ is not minimal.

Conversely, if $\aG$ is not minimal, then there is an isotropic, almost simple, proper $\rational$-subgroup~$\aH$ of~$\aG$, such that $\Rrank \aH \ge 2$. Since $\aG$ is isomorphic to $\aSL_3 \times \aSL_3$ over~$\overline{\rational}$, we know that $\aH$ must be isogenous to either $\aSL_3$ or $\aSL_2 \times \aSL_2$ over~$\overline{\rational}$. In either case, because $\Rrank \aH \ge 2$, there is a real quadratic extension~$F$ of~$\rational$, such that $\fieldrank{F} \aH = 2$. Therefore $\fieldrank{F} \aG \ge 2$.
If $F \not\subset L$, then $\tau$ extends to an automorphism of $L \cdot F$ that is trivial on $K \cdot F$, and $\aG$ is $F$-isomorphic to 
	$$\res{K \cdot F/F} \aSU_3(L \cdot F, f, \tau) .$$
So $\fieldrank{F} \aG = 1$. This is a contradiction, so we conclude that $L$ does contain the real quadratic extension~$F$.

\item Let $\aG$ be one of the groups described in  \fullref{QgrpExplicitThm}{SL2}. Any $\rational$-subgroup of~$\aG$ that is almost simple is isogenous to $\res{L/\rational} \aSL_2$, for some subfield~$L$ of~$K$. Thus, $\aG$ fails to be minimal if and only if $\Rrank \bigl(  \res{L/\rational} \aSL_2 \bigr) > 1$ for some proper subfield~$L$ of~$K$.
\end{itemize}

\subsection*{Justification of \S\ref{LatticeSubsec}}
The Margulis Arithmeticity Theorem \cite[Thm.~IX.1.16, p.~299, and Rem.~IX.1.6(iii), p.~294]{MargulisBook} states that (up to isomorphism of finite-index subgroups) the collection of nonuniform lattices of higher rank is the same as
	$$ \{\, \aG(\integer) \mid \text{$\aG$ is an isotropic almost simple $\rational$-group with $\Rrank \aG \ge 2$} \,\} .$$
For isotropic almost simple $\rational$-groups $\aG$ and~$\aG_1$ with $\Rrank \aG_1 \ge 2$,
the Margulis Superrigidity Theorem \cite[Thm.~IX.5.12(ii), p.~327, and Rem.~IX.1.6(iv), p.~295]{MargulisBook} implies there is a finite-index subgroup of $\aG_1(\integer)$ that is isomorphic to a subgroup of $\aG(\integer)$ if and only if $\aG_1$ is isogenous to a subgroup of $\aG$. Hence, $\aG(\integer)$ is almost minimal (as a nonuniform lattice of higher rank) if and only if $\aG$ is minimal (as an algebraic $\rational$-group). Therefore, all the assertions of \S\ref{LatticeSubsec} are simply translations of results in \S\ref{QgroupSubsec}. For example, because $\aSL_3$ is minimal, it is immediate that $\SL_3(\integer)$ is almost minimal.

\subsection*{Justification of \S\ref{SymmspaceSubsec}}
Let $X$ be as in Theorem~\ref{LocSymmThm}. It is well known (cf.\ \cite[\S2.2, pp.~70--71]{EberleinBook} and \cite[Thm.~5.6, p.~222]{HelgasonBook}) that, up to isometry, we have $X = \Gamma \backslash G/K$, where 
\begin{itemize}
\item $G$ is a connected, semisimple, adjoint Lie group with no compact factors,
\item $K$ is a maximal compact subgroup of~$G$,
and
\item $\Gamma$ is a (torsion-free) nonuniform, irreducible lattice in~$G$.
\end{itemize}
We have $\Rrank G = \rank X$ (cf.\ \cite[\S2.7, pp.~76--77]{EberleinBook}), so, since $\rank X \ge 2$, we see that $\Gamma$ is a nonuniform lattice of higher rank. Hence, Theorem~\ref{LattThm} implies that $\Gamma$ contains a subgroup~$\Gamma'$ that is isomorphic to a nonuniform, irreducible lattice in a connected, semisimple, adjoint Lie group~$H$, and $H$ is locally isomorphic to either $\SL_3(\real)$, $\SL_3(\complex)$, or a product $\SL_2(\real)^m \times \SL_2(\complex)$, with $m + n \ge 2$. The Margulis Superrigidity Theorem \cite[Thm.~IX.5.12, p.~327]{MargulisBook} implies that (after passing to a finite-index subgroup of~$\Gamma'$), the inclusion $\Gamma' \hookrightarrow \Gamma$ extends to an embedding $H \hookrightarrow G$, so we may assume $H \subset G$ and $\Gamma' = H \cap \Gamma$. We may choose a Cartan involution~$\sigma$ of~$G$, such that $\sigma(H) = H$ \cite[Thm.~7.3]{MostowSelfAdjoint}. Let
	\begin{itemize}
	\item $g_0 \in G$, such that $g_0 K g_0^{-1}$ is the maximal compact subgroup of~$G$ on which $\sigma$ is trivial \cite[Thm.~2.2(i), p.~256]{HelgasonBook}, 
	\item $K' = (g_0 K g_0^{-1}) \cap H$, so $K'$ is a maximal compact subgroup of~$H$,
	and
	\item $X' = \Gamma' \backslash H / K'$.
	\end{itemize}
Then $X'$ is a a finite-volume, noncompact, irreducible locally symmetric space whose universal cover is $H/K'$. The immersion
	$$ X' \to X \colon \Gamma' h K' \mapsto \Gamma h g_0 K $$
is proper \cite[Thm.~1.13, p.~23]{RaghunathanBook} and has totally geodesic image
\cite[Prop~2.6.2, p.~74]{EberleinBook}.

\section{Preliminaries} \label{PrelimSect}

Throughout the remainder of this paper, $\aG$ is a connected, isotropic, almost simple algebraic group over an algebraic number field~$F$. 

\begin{rem}
Our notation and terminology for discussing algebraic groups generally follows \cite{[PR]}.
However, we use boldface letters ($\aG$, $\aH$, $\aT$, etc.) to denote algebraic groups. Also, if $A$ is a central simple algebra, and $f$ is a Hermitian (or skew-Hermitian) form on~$A^m$, with respect to an involution~$\tau$, we use $\aSU_m(A,f,\tau)$ to denote the corresponding special unitary group, whereas \cite{[PR]} writes merely $\aSU_m(f)$. 
\end{rem}

\begin{notation}
Let $\SG$ be the set of all archimedean places~$v$ of~$F$, such that $\locrank \aG \ge 2$. 
\end{notation}

\begin{defn}
We say $\aG$ is \emph{minimal} if $\SG \neq \emptyset$, and there does not exist a proper, isotropic, almost simple $F$-subgroup~$\aH$ of~$\aG$, such that $\locrank \aH \ge 2$ for every $v \in \SG$.
\end{defn}

The following is a natural generalization of Theorem~\ref{QgrpExplicitThm}.

\begin{thm} \label{MainForF}
 Suppose\/ $\aG$ is an isotropic, almost simple algebraic group over an algebraic number field~$F$, such that $\SG \neq \emptyset$. If\/ $\aG$ is minimal, then\/ $\aG$ is isogenous to either:
\begin{enumerate} \renewcommand{\theenumi}{\roman{enumi}}
\item \label{MainForF-SL3}
$\aSL_3$,
or
\item \label{MainForF-SU}
 $\aSU_3(L,f,\tau)$, where 
 \begin{itemize}
 \item $L$ is a quadratic extension of~$F$, such that $L \subset F_v$, for some archimedean place~$v$ of~$F$,
 \item $\tau$~is the Galois automorphism of~$L$ over~$F$, 
 and 
\item $f(x_1,x_2,x_3) = \tau(x_1) \, x_1 - \tau(x_2) \, x_2 - \tau(x_3) \, x_3 $,
\end{itemize}
or
\item \label{MainForF-SUres}
 $\res{K/F} \aSU_3(L,f,\tau)$, where 
 \begin{itemize}
 \item $K$ is a quadratic extension of~$F$, such that $K \not\subset F_v$, for some archimedean place~$v$ of~$F$,
 \item $L$~is a quadratic extension of~$K$, 
 \item $\tau$~is the Galois automorphism of~$L$ over~$K$, 
 and 
 \item $f(x_1,x_2,x_3) = \tau(x_1) \, x_1 - \tau(x_2) \, x_2 - \tau(x_3) \, x_3 $,
 \end{itemize}
or
\item \label{MainForF-SL2}
 $\res{K/F} \aSL_2$, for some nontrivial finite extension~$K$ of~$F$, such that either $|K:F| > 2$, or $K \subset F_v$, for some archimedean place~$v$ of~$F$.
\end{enumerate}
 \end{thm}

\begin{cor}
 Suppose\/ $\aG$ is an isotropic, almost simple algebraic group over an algebraic number field~$F$, such that $\SG \neq \emptyset$. Then $\aG$ contains an isotropic, almost simple $F$-subgroup~$\aH$, such that $\locrank \aH \ge 2$ for every $v \in \SG$, and $\aH$ is isogenous to a subgroup described in \pref{MainForF-SL3}, \pref{MainForF-SU}, \pref{MainForF-SUres}, or~\pref{MainForF-SL2} of Theorem~\ref{MainForF}.
\end{cor}

The remainder of this paper provides a proof of Theorem~\ref{MainForF}.

\begin{notation}
For algebraic groups~$\aG_1$ and~$\aG_2$ over a field~$K$, we write $\aG_1 \approx \aG_2$ if they have the same simply connected covering.
\end{notation}

Let us record an observation that will be used repeatedly.

\begin{lem} \label{SO=SL_2(F[a])}
 If $a \in F^*$ and $\sqrt{a} \notin F$, then 
 	$$\aSO_4(x_1^2 - x_2^2 - x_3^2 + a x_4^2) \approx \res{F[\sqrt{a}]/F} \aSL_2 .$$
 \end{lem}

\begin{proof} 
$\aSO_4$ is of type $D_2 = A_1\times A_1$. Since the
discriminant of the quadratic form under consideration is not a
square, we know that the associated orthogonal group is an outer
form. Thus, it is isogenous to $\res{F[\sqrt{a}]/F} \aSL_1(A)$, where
$A$ is a quaternion
algebra over $F[\sqrt{a}]$. Since the group is isotropic over $F$,
the algebra $A$ must be split, so $\aSL_1(A) \iso \aSL_2$.
 \end{proof}

Recall that a connected algebraic $F$-group is \emph{absolutely almost simple} if it remains simple over an algebraic closure~$\overline{F}$ of~$F$. The following basic observations allow us to assume that $\Frank \aG = 1$, and that $\aG$ is absolutely almost simple.

\begin{lem} \label{FrankMust1}
If\/ $\aG$ is minimal, then either:
\begin{enumerate}
\item \label{FrankMust1-Frank1}
 $\Frank \aG = 1$,
 or
\item \label{FrankMust1-SL3}
$\aG$ is isogenous to $\aSL_3$ {\rm(}so \fullref{MainForF}{SL3} holds\/{\rm)}.
\end{enumerate}
\end{lem}

\begin{proof}
Assume $\Frank \aG \ge 2$. It is well known that $\aG$ contains an $F$-subgroup that is isogenous to either $\aSL_3$ or $\aSp_4$ \cite[Prop.~I.1.6.2, p.~46]{MargulisBook}. By minimality, $\aG$ itself must be isogenous to either $\aSL_3$ or $\aSp_4$.

Suppose $\aG$ is isogenous to $\aSp_4$. Then $\aG$ is a split group of type $C_2 = B_2$, so it is also isogenous to
	$$ \aSO_5( x_1^2 - x_2^2 - x_3^2 + x_4^2 + a x_5^2) ,$$
for any $a \in F$.
It therefore contains a subgroup isogenous to
	$$ \aSO_4( x_1^2 - x_2^2 - x_3^2 + a x_5^2) .$$
By Weak Approximation, we may choose $a$ so that $a$ is a square in~$F_v$, for every $v \in \SG$, but $a$ is not a square in~$F$. Then $\aH$ is isogenous to $\res{F[\sqrt{a}]/F} \aSL_2$ \see{SO=SL_2(F[a])}, so it is isotropic and $\locrank \aH  = 2$ for every $v \in \SG$. This contradicts the minimality of~$\aG$.
\end{proof}

\begin{lem} \label{AbsSimple}
If\/ $\aG$ is minimal, then either:
\begin{enumerate}
\item \label{AbsSimple-SL}
 $\aG$ is isogenous to $\res{K/F} \aSL_2$, with $K$ as described in Theorem~\fullref{MainForF}{SL2},
or
\item \label{AbsSimple-Simple}
 $\aG$ is absolutely almost simple, 
or
\item \label{AbsSimple-Imag}
 $\aG$ is isogenous to $\res{K/F}\aG_0$, where $\aG_0$ is an absolutely almost simple group over a quadratic extension~$K$ of~$F$, such that $K \not\subset F_v$, for some $v \in \SG$.
\end{enumerate}
\end{lem}

\begin{proof}
Assume \pref{AbsSimple-Simple} does not hold. Then there is an algebraic number field $K \supset F$, and an absolutely almost simple group~$\aG_0$ over~$K$, such that $\aG$ is isogenous to $\res{K/F} \aG_0$ \cite[Thm.~26.8, p.~365]{BookInvols}. Since $\aG$ is isotropic over~$F$, we know $\aG_0$ is isotropic over~$K$, so $\aG_0$ contains a subgroup that is isogenous to $\aSL_2$. Therefore, $\aG$ contains a subgroup~$\aH$ that is isogenous to $\res{K/F} \aSL_2$. 

If $\locrank \aH \ge 2$, for every $v \in \SG$, then the minimality of~$\aG$ implies $\aG = \aH$, so \pref{AbsSimple-SL} holds. On the other hand, if $\locrank \aH = 1$, for some $v \in \SG$, then $K$ is a quadratic extension of~$F$, and $K \not\subset F_v$, so \pref{AbsSimple-Imag} holds.
\end{proof}

\begin{lem} \label{CanAbsSimple}
If Theorem~\ref{MainForF} holds\/ {\rm(}for all algebraic number fields\/{\rm)} under the additional assumption that\/ $\aG$ is absolutely almost simple, then it holds in general.
\end{lem}

\begin{proof}
Suppose $\aG$ is minimal, but is not absolutely almost simple. From Lemma~\ref{FrankMust1}, we see that $\Frank \aG = 1$. We may assume \fullref{AbsSimple}{Imag} holds (for otherwise \fullref{MainForF}{SL2} holds). Since $\aG$ is minimal (as an $F$-group), it is clear that $\aG_0$ is minimal (as a $K$-group). The only absolutely almost simple group of global rank~$1$ in the conclusion of Theorem~\ref{MainForF} is in \fullref{MainForF}{SU}. Thus, we conclude that $\aG_0$ is as described in \fullref{MainForF}{SU}, but with $F$ replaced by~$K$. Then $\aG = \res{K/F} \aG_0$ is as described in \fullref{MainForF}{SUres}.
\end{proof}

Lemma~\ref{FrankMust1} immediately rules out some types of exceptional groups:

\begin{cor} \label{MostExceptCor}
If\/ $\aG$ is minimal, then\/ $\aG$ is not of type $E_7$, $E_8$, or~$G_2$.
\end{cor}

\begin{proof}
The Tits Classification \cite[pp.~59--61]{[T66]} shows there are no rank~1 forms of any of these types over a number field.
\end{proof}

The following useful observation is well known, and easy to prove.

\begin{lem} \label{NormalizeHerm}
Let
\begin{itemize}
\item $D$ be a quaternion algebra over a field~$L$,
\item $\tau$ be an involution of~$D$ {\rm(}of either the first or second kind{\rm)},
\item $f(x,y) = \tau(x_1) \, a_1 \, y_1 + \tau(x_2) \, a_2 \, y_2 + \cdots + \tau(x_n) \, a_n \, y_n$ be a nondegenerate $\tau$-Hermitian form on~$D^n$, for some~$n$,
\item $d \in D$, such that $\tau(d) = d$,
\item $\tau' = \int(d) \circ \tau$, where $\int(d)$ is the inner conjugation in~$D$ by~$d$,
and
\item $f'(x,y) = d \, f(x,y) = \tau'(x_1) \, d a_1 \, y_1 + \tau'(x_2) \, d a_2 \, y_2 + \cdots + \tau'(x_n) \, d a_n \, y_n$.
\end{itemize}
Then:
\begin{enumerate}
\item $\tau'$ is an involution {\rm(}of the same kind as~$\tau${\rm)},
\item $f'$ is $\tau'$-Hermitian,
and
\item $\aSU_n(D, f', \tau') = \aSU_n(D, f, \tau)$.
\end{enumerate}
\end{lem}

\begin{defn}[{\cite[\S2.2, p.~69]{[T66]}}]
Recall that if $\aS$ is a maximal $F$-split torus in~$\aG$,
then the semisimple part of the centralizer $\aC_\aG(\aS)$ is called the
\emph{semisimple $F$-anisotropic kernel} of~$\aG$. It is unique up
to $F$-isomorphism.
\end{defn}

\begin{defn} 
A connected, semisimple subgroup~$\aH_0$ of~$\aG$ is 
 \emph{standard} if $\aH_0$ is normalized by a maximal torus~$\aT$
 of~$\aG$. (We remark that neither~$\aH_0$ nor~$\aT$ is assumed
 to be defined over~$F$.) Equivalently, there exist
 roots $\beta_1,\ldots,\beta_r$ of~$\aG$ (with respect to~$\aT$),
 such that $\aH_0$ is generated by the root subgroups 
 $U_{\pm \beta_1}, \ldots, U_{\pm \beta_r}$. 
 For short, we may say that $\aH_0$ is \emph{generated by the 
 roots $\pm\beta_1,\ldots,\pm\beta_r$}.
 \end{defn}

The following useful observation is well known (cf.\ \cite[pp.~353]{[PR]}).

\begin{prop} \label{Malpha}
Let
\begin{itemize}
\item $\aM$ be an anisotropic, semisimple group over~$F$, such that $-1$ is in the Weyl group of\/~$\aM$,
\item $L$ be a quadratic extension of~$F$, such that\/ $\aM$ is quasisplit over~$L$,
and
\item $\alpha$ be a simple root of\/~$\aM$ that is fixed in the $*$-action shown in the Tits index of\/~$\aM$. 
\end{itemize}
Then there is a maximal $F$-torus\/ $\aT$ of\/~$\aM$, such that the standard subgroup\/~$\aM_{\alpha}$ generated by the roots $\pm \alpha$ is defined over~$F$.

Furthermore, if\/ $\aM$ is split over~$L$, then\/ $\aT$ may be chosen to be split over~$L$.
\end{prop}

\begin{proof}
Letting $\sigma$ be
the Galois automorphism of~$L$ over~$F$, there is a Borel
$L$-subgroup~$\aB$ of~$\aM$, such that $\aT = \aB \cap \sigma(\aB)$ is a
maximal torus of~$\aM$ \cite[Lem.~6.17, p.~329]{[PR]}. 
The Borel subgroup~$\aB$ determines an ordering of the roots of~$\aM$
 (with respect to~$\aT$). Note that the negative roots are precisely those
 that appear in $\sigma(\aB)$.
 
Let $K$ be a Galois splitting field of~$\aM$ that contains~$L$.
Since $\aT$ is defined over~$F$, the Galois group $\Gal(K/F)$
permutes the root spaces of~$\aM$. Furthermore, for any $\tau \in \Gal(K/F)$,
either $\tau$ sends every positive root to a positive root (if $\tau(\aB) = \aB$),
or $\tau$ sends every positive root to a negative root (if $\tau(\aB) = \sigma(\aB)$).
Since $\alpha$ is fixed in the 
$*$-action shown in the Tits index, and $-1$~belongs to the
Weyl group, 
 this implies that $\tau(\alpha) = \pm \alpha$. 
Therefore, $\aM_{\alpha}$  is stable 
under $\Gal(K/F)$; thus, it is defined over~$F$.
\end{proof}

It is easy to tell whether a standard subgroup of a simply connected group is simply connected:

\begin{rem}[{}{\cite[(II.5.3), p.~206]{SpringerSteinberg}}] \label{WhichSC}
Let $\aG$ be a simply connected, semisimple $F$-group, and let $\aH$ be the standard, semisimple subgroup of~$\aG$ generated by the roots $\pm\beta_1,\ldots,\pm\beta_r$. Then $\aH$ is simply connected if and only if the set of roots of~$\aH$ contains every long root of~$\aG$ that is in the $\rational$-span of $\{\beta_1,\ldots,\beta_r\}$.
\end{rem}

\section{Groups of classical type} \label{ClassicalSect}

\begin{assump}
We assume in this section that $\aG$ is a group of classical type, and that $\aG$ is minimal. Furthermore, with Lemmas~\ref{FrankMust1} and~\ref{CanAbsSimple} in mind, we assume that $\Frank \aG = 1$ and that $\aG$ is absolutely almost simple.
\end{assump}

We know $\aG \neq \aSp_n$ (because $\aSp_n$ is $F$-split, but $\Frank \aG = 1 < \locrank \aG$ for any $v \in \SG$).
Thus, $\aG$ is either a special linear group, an orthogonal group, or a unitary group of either the first or second kind \cite[\S2.3.4]{[PR]}.
We consider each of these possibilities separately.

\subsection{Special Linear Groups}
 \label{SLSect}

\begin{assumps} \ 
 \begin{itemize}
 \item $D$ is a central division algebra over~$F$,
 and
 \item $\aG = \aSL_2(D)$.
 \end{itemize}
 \end{assumps}

Let $K$ be a maximal subfield of~$D$. For $v \in \SG$, we have  
$$\locrank \aG > 1 = \locrank \aSL_2 .$$
Therefore $D \neq F$, so $K$ is a proper extension of~$F$.  Because $\res{K/F} \aSL_2 \subseteq \aG$, the minimality of~$\aG$ implies there exists $w \in \SG$, such that $\fieldrank{F_w} ( \res{K/F} \aSL_2 ) = 1$. Therefore $|K:F| = 2$, so $D$ is a quaternion algebra over~$F$.

Write $D = (a,b)_F$. 
By Weak Approximation, there exist $c_1,c_2,c_3 \in F$, such that
	\begin{equation} \label{SLSect-allpos}
	 \text{for every $v \in \SG$, $a c_1^2 + b c_2^2 -ab c_3^2$ is a nonzero square in~$F_v$.} 
	 \end{equation}
Let $c = a c_1^2 + b c_2^2 - ab c_3^2 \in F$, so $c$ has a square root in~$D$.
Thus, letting $\aH = \res{F[ \sqrt{c}]/F} \aSL_2$, we have 
$\aH \subseteq \aG$.
Also, for every $v \in \SG$, we know $c$ is a square in~$F_v$ \see{SLSect-allpos}, so $\locrank \aH \ge 2$. 
This contradicts the minimality of~$\aG$.

\subsection{Orthogonal groups}
\label{orthogonal}

\begin{assumps} \ 
\begin{itemize}
\item $f$ is a nondegenerate quadratic form on~$F^n$, for some~$n \ge 5$,
\item $\aG = \aSO_n(f)$,
and
\item the maximal totally isotropic $F$-subspace of~$F^n$ is $1$-dimensional
(in other words, $\Frank \aG = 1$).
\end{itemize}
\end{assumps}

 After a change of basis, to diagonalize the form, we may write
 $$ f(x) = x_1^2 - x_2^2 + a_3 x_3^2 + a_4 x_4^2
  + \cdots + a_n x_n^2  .$$
 (We may assume the form begins with $x_1^2 - x_2^2$, because it is isotropic.)
By normalizing the form, we may assume $a_3 = -1$. 
By Weak Approximation, there exist $b_4,b_5,\ldots,b_n \in F$, such that
	\begin{equation} \label{orthogonal-allpos}
	 \text{for every $v \in \SG$, $a_4 b_4^2 + \cdots + a_n b_n^2$ is a nonzero square in~$F_v$.} 
	 \end{equation}
Let $a = a_4 b_4^2 + \cdots + a_n b_n^2$, so,
after a change of basis, we may assume $a_4 = a$. Then
	$$ \aH =  \aSO_4(x_1^2 - x_2^2 - x_3^2 + a x_4^2) \subset \aG. $$
For any $v \in \SG$, we know
	$\locrank \aH = 2$
(since $a$ is a square in~$F_v$). Hence, the minimality of~$\aG$ implies 
	$$\aG = \aH \approx \res{F[\sqrt{a}]/F} \aSL_2 . $$
So \fullref{MainForF}{SL2} holds.

\subsection{Unitary groups of the second kind}
 \label{2ndKindSect}

\begin{assumps} \ 
 \begin{itemize}
 \item $L$ is a quadratic extension of~$F$,
 \item $D$ is a central division algebra over~$L$,
 \item $\tau$ is an involution of~$D$ that fixes every element of~$F$, but fixes no other elements of~$L$,
 \item $f$ is a $\tau$-Hermitian form on $D^n$, for some~$n$,
 \item $\aG = \aSU_n(D,f,\tau)$,
 and
 \item the maximal totally isotropic $D$-subspace of~$D^n$
is $1$-dimensional (in other words, $\Frank \aG = 1$).
 \end{itemize}
 \end{assumps}

 After a change of basis, to diagonalize the form, we may write
 $$ f(x,y) = x_1^\tau y_1 - x_2^\tau y_2 + x_3^\tau a_3 y_3 + x_4^\tau
a_4 y_4 + \cdots + x_n^\tau a_n y_n  ,$$
 where $a_j^\tau = a_j$ for each~$j$.
 (We may assume the form begins with $x_1^\tau y_1 - x_2^\tau y_2$, because it is isotropic.)

\setcounter{case}{0}

\begin{case}
 Assume $D = L$.
 \end{case}
 By normalizing the form $f(x,y)$, we may assume $a_3 = -1$.
 We may also assume $n \ge 4$; otherwise \fullref{MainForF}{SU} holds.
  
 For each $v \in \SG$:
 \begin{itemize}
 \item let $L_v = F_v \otimes_F L$,
 \item identify $F_v$ with $F_v \otimes_F F \subset L_v$,
 and
 \item let $\tau_v$ be the extension of~$\tau$ to an involution of~$L_v$ with fixed field~$F_v$.
 \end{itemize}
 
 \begin{claim}
 For every $v \in \SG$, there exist 
 $b_{v,4},b_{v,5},\ldots,b_{v,n} \in L_v$, such that 
 $$ \text{$a_4 \, b_{v,4}^{\tau_v} \, b_{v,4} + a_5 \, b_{v,5}^{\tau_v} \, b_{v,5} + \cdots + a_n \, b_{v,n}^{\tau_v} \, b_{v,n}$ is a nonzero square in~$F_v$.} $$
 \end{claim}
We consider two possibilities:
\begin{itemize}
 \item If $L \not\subset F_v$, then $L_v$ is a field extension of~$F_v$, and $\aG$ is isomorphic over~$F_v$ to
 $$\aSU_n(L_v, x_1^{\tau_v} y_1 - x_2^{\tau_v} y_2 - x_3^{\tau_v} y_3 
 + a_4 x_4^{\tau_v} y_4 + a_5 x_5^{\tau_v} y_5 + \cdots + a_n x_n^{\tau_v} y_n, \tau_v) .$$
 The desired conclusion follows from the fact that $\locrank \aG \ge 2$.
 \item If $L \subset F_v$, then there is an isomorphism $\varphi_v \colon (L_v,\tau_v) \to (F_v \oplus F_v, \overline{\tau})$, where $\overline{\tau}(x_1,x_2) = (x_2,x_1)$. Since
 	$ (x,1)^{\overline{\tau}} \, (x,1) = (x,x)$ is an arbitrary element of~$\varphi_v(F_v)$, 
the desired conclusion is obvious.
\end{itemize}
This completes the proof of the claim.

\medskip 

Combining the above claim with Weak Approximation yields
$b_4,b_5,\ldots,b_n \in L$, such that, for every $v \in \SG$,
 $$ \text{$a_4 \, b_4^{\tau_v} \, b_4 + a_5 \, b_5^{\tau_v} \, b_5 + \cdots + a_n \, b_n^{\tau_v} \, b_n$ is a nonzero square in~$F_v$.} $$
Let $a = a_4 \, b_4^{\tau_v} \, b_4 + a_5 \, b_5^{\tau_v} \, b_5 + \cdots + a_n \, b_n^{\tau_v} \, b_n$, so, after a change of basis, we may assume $a_4 = a$. Then
	$$ \aH = \aSO_4(x_1^1 - x_2^2 - x_3^2 + a x_4^2) \subset \aG .$$
From the choice of~$a$, we know $\locrank \aH = 2$ for every $v \in \SG$.
(Also, since $\Frank \aG = 1$, we know that $a = a_4$ is not a square in~$F$, so 
	$ \aH \approx \res{F[\sqrt{a}]/F} \aSL_2 $
is almost simple.) This contradicts the minimality of~$\aG$.

\begin{case} \label{2ndKindSect-DnotL}
 Assume $D \neq L$.
 \end{case}
 Choose a maximal subfield~$K$ of~$D$, such that $K$ is invariant
under~$\tau$, and let $K_0$ be the fixed field of~$\tau$ in~$K$. Then
	\begin{equation} \label{2ndKindSect-DnotL-GhasSL2}
 \aG
 \supset \res{K_0/F} \aSU_2(K,x_1^\tau y_1 - x_2^\tau y_2, \tau|_K) 
 \approx \res{K_0/F} \aSL_2
 	 . \end{equation}
Thus, we may assume $K_0$ is a quadratic extension of~$F$,
for otherwise minimality implies \fullref{MainForF}{SL2} holds.
 Then, since $|L:F| = 2 = |K:K_0|$, we have 
 	$$ 2 = |K_0:F| 
	= \frac{ |K:L| \cdot |L : F|}{|K:K_0|}
	= |K:L| ,$$
so $D$ is a quaternion algebra.

There is a quaternion algebra $D'$ over~$F$, such that $D = D' \otimes_F L$, and $\tau|_{D'}$ is the canonical involution \cite[Thm.~11.2(ii), p.~314]{Scharlau}. 

\begin{subcase}
Assume $n = 2$.
\end{subcase}
For every $v \in \SG$, we know $F_v$ splits~$D$ (because $n = 2$ and $\locrank \aG \ge 2$). 
Therefore, by Weak Approximation and the Hasse Principle, there is a quadratic extension~$E$ of~$F$, such that $E$ splits~$D'$ and, for each $v \in \SG$,
	\begin{equation} \label{2ndKindSect-DnotL-EinFv<>LinFv}
	E \subset F_v \quad \iff \quad L \subset F_v 
	. \end{equation}
Then $D$ splits over $E \cdot L$, so we may assume $K = E \cdot L$. 
Since $\tau$ is nontrivial on both $E$ and~$L$, we see from \pref{2ndKindSect-DnotL-EinFv<>LinFv} that the fixed field 
$K_0$ of~$\tau$ is contained in~$F_v$, for every $v \in \SG$. So the minimality of~$\aG$ (together with \pref{2ndKindSect-DnotL-GhasSL2}) implies \fullref{MainForF}{SL2} holds.

\begin{subcase}
Assume $n \ge 3$.
\end{subcase}
By replacing $\tau$ with $\int(a_3^{-1}) \circ \tau$, we may assume $a_3 = 1$ \cf{NormalizeHerm}.
By Weak Approximation and the Hasse Principle, there is a quadratic extension~$E$ of~$F$, such that $E$ splits~$D'$ and, for each $v \in \SG$,
	\begin{equation} \label{2ndKindSect-DnotL-LnotinFv->EnotinFv}
	L \not\subset F_v \quad\implies\quad E \not\subset F_v
	. \end{equation}
Then $D$ splits over $E \cdot L$, so we may assume $K = E \cdot L$. 

Let $\aH_0 = \aSU_3(K, x_1^\tau y_1 - x_2^\tau y_2 + x_3^\tau y_3 , \tau|_K)$ and $\aH = \res{K_0/F} \aH_0 \subset \aG$. For any $v \in \SG$:
\begin{itemize}
\item If $K_0 \subset F_v$, then $K_0 \otimes_F F_v \iso F_v \oplus F_v$. Therefore, it is clear that $\locrank \aH \ge 2$ (since $\aH_0$ is isotropic).
\item If $K_0 \not\subset F_v$, then, from \pref{2ndKindSect-DnotL-LnotinFv->EnotinFv} and the fact that $\tau$ is nontrivial on both $E$ and~$L$, we see that $L \subset F_v$. Therefore $\aH_0$ is inner (hence, split) over the field $K_0 \otimes_F F_v$, so $\locrank \aH \ge 2$.
\end{itemize}
This contradicts the minimality of~$\aG$.

\subsection{Unitary groups of the first kind}

\begin{assumps} \ 
 \begin{itemize}
 \item $D$ is a quaternion algebra over~$F$,
 \item $\tau$ is the canonical involution of~$D$,
 \item $f$ is a $\tau$-Hermitian or $\tau$-skew Hermitian form on $D^n$,
for some~$n$,
 \item $\aG = \aSU_n(D, f, \tau)$,
 and
 \item the maximal totally isotropic $D$-subspace of~$D^n$
is $1$-dimensional (in other words, $\Frank \aG = 1$).
 \end{itemize}
 \end{assumps}

 After a change of basis, to diagonalize the form, we may write
 $$ f(x,y) = 
 \begin{cases}
 x_1^\tau y_1 - x_2^\tau y_2 + x_3^\tau a_3 y_3 + x_4^\tau
a_4 y_4 + \cdots + x_n^\tau a_n y_n 
 & \text{if $f$ is Hermitian}, \\
 x_1^\tau y_2 - x_2^\tau y_1 + x_3^\tau a_3 y_3 + x_4^\tau
a_4 y_4 + \cdots + x_n^\tau a_n y_n 
 & \text{if $f$ is skew-Hermitian} 
 . \end{cases} $$
 (We may assume the form begins with $x_1^\tau y_1 - x_2^\tau y_2$ or
$x_1^\tau y_2 - x_2^\tau y_1$ respectively, because  $\Frank \aG = 1 \neq
0$.)
 Note that $a_3,\ldots,a_n$ are
 \begin{itemize}
 \item elements of~$F$ if $f$ is Hermitian,
 and
 \item purely imaginary if $f$ is skew-Hermitian.
 \end{itemize}

\setcounter{case}{0}

\begin{case}
Assume $n \le 3$.
\end{case}
The quaternion algebra~$D$ must split over~$F_v$, for each $v \in \SG$ (because $\locrank \aG \ge 2$).

\begin{subcase}
 Assume $f$ is Hermitian.
 \end{subcase}
 Let
 $$ \aG' =  \aSU_2(D, x_1^\tau y_1 - x_2^\tau y_2, \tau) \subset \aG .$$
 Then $\aG'$ is of type~$\aC_2$ (see \cite[Prop.~2.15(2), p.~86]{[PR]}), so it is also of type~$B_2$. Therefore, it has a realization to which \S\ref{orthogonal} applies.

\begin{subcase}
 Assume $f$ is skew-Hermitian.
 \end{subcase}
 Because $\aG$ is absolutely almost simple, we know it is not of type $D_2 = A_1 \times A_1$. Therefore $n = 3$. Then $\aG$ is of type~$D_3$, so it is also of type~$A_3$. Therefore, it
has a realization to which either \S\ref{SLSect} or~\S\ref{2ndKindSect}
applies.

\begin{case}
 Assume $n \ge 4$.
 \end{case}

\begin{subcase}
 Assume $f$ is Hermitian.
 \end{subcase}
By Weak Approximation and the Hasse Principle, there is a quadratic extension~$E$ of~$F$, such that $E$ splits~$D$ and, for $v \in \SG$,
	\begin{equation} \label{1stKindSect-n>4-Herm-E}
	\text{$F_v$ splits $D$} \quad\implies\quad E\subset F_v
	. \end{equation}

By normalizing, we may assume $a_3 = -1$. 
By Weak Approximation, there exist $b_4,b_5,\ldots,b_n \in F$, with the property that, for every $v \in \SG$, such that $F_v$ does \emph{not} split~$D$, we have 
	\begin{equation} \label{1stKindSect-n>4-Herm-a}
	\text{$a_4 b_4^2 + a_5 b_5^2 + \cdots + a_n b_n^2 > 0$ in~$F_v$.} 
	\end{equation}
Let $a = a_4 b_4^2 + a_5 b_5^2 + \cdots + a_n b_n^2$, so, after a change of basis, we may assume $a_4 = a$.

Let 
	$$ \aH = \aSU_4( E, x_1^\tau y_1 - x_2 y_2^\tau - x_3 y_3^\tau + a x_4^\tau y_4, \tau|_E) \subset \aG .$$
For any $v \in \SG$:
\begin{itemize}
\item If $E \subset F_v$, then $\aH$ is split over~$F_v$, so $\locrank \aH = 3$.
\item If $E \not\subset F_v$, then $F_v$ does not split~$D$ \see{1stKindSect-n>4-Herm-E}, so $a > 0$ in~$F_v$ \see{1stKindSect-n>4-Herm-a}. Hence $\locrank \aH = 2$.
\end{itemize}
This contradicts the minimality of~$\aG$.

\begin{subcase}
 Assume $f$ is skew-Hermitian.
 \end{subcase}
 Because $f$ is skew-Hermitian, we know that $a_3$ and~$a_4$ are
purely imaginary elements of~$D$, so there exists a nonzero, purely
imaginary $\alpha \in D$, such that $a_3$~and~$a_4$ both negate~$\alpha$;
that is, $a_3 \alpha = - \alpha a_3$ and $a_4 \alpha = - \alpha a_4$. (To
see this, note that $a_3$ and~$a_4$ each negate a $2$-dimensional space of
imaginary elements of~$D$. Since the imaginary elements form only a
$3$-dimensional space, there must be nonzero intersection.) Hence, $a_3$
and~$a_4$ act by conjugation on $F[\alpha]$.

Let
 \begin{itemize}
 \item $F' = F[\alpha] \subset D$,
 \item $\{e_1, e_2, e_3, e_4\}$ be an orthogonal basis of~$D^4$, such that
 $$ \text{$f(e_1,e_1) = -f(e_2,e_2) = f(e_3,e_3) = a_3$ and $f(e_4,e_4) =
a_4$} $$
 (namely, $e_1 = \frac{1}{2}(a_3,1,0,0)$, $e_2 =
\frac{1}{2}(-a_3,1,0,0)$, $e_3 = (0,0,1,0)$, and $e_4 = (0,0,0,1)$),
 \item $V'$ be the $F'$ span of $\{e_1, e_2, e_3, e_4\}$,
 and
 \item $f'$ be the restriction of $a_3^{-1} f$ to~$V'$.
 \end{itemize}
 Note that $a_3^{-1} a_4$ centralizes~$\alpha$, so it must belong
to~$F'$.

Then 
 \begin{itemize}
 \item $f'(V' \times V') \subseteq F'$, 
 \item $f'$ is a nondegenerate, symmetric $F'$-bilinear form on~$V'$,
  \item $f'$ is isometric to the form $f'' = x_1 y_1 - x_2 y_2 + x_3 y_3 +
(a_3^{-1} a_4) x_4 y_4$ on~$(F')^4$,
and
 \item $\aG
 =\aSU_n(D,f,\tau) 
 \supset \res{F'/F} \aSO(f') 
 \approx \res{K/F} \aSL_2$,
 where $K = F'\left[ \sqrt{-a_3^{-1} a_4} \, \right]$.
 \end{itemize}
Since $\Frank \aG = 1$, we know that $f'$ has no $2$-dimensional totally
isotropic subspace, so $-a_3^{-1} a_4$ is \emph{not} a square in~$F'$.
Hence, we have
 $$|K : F|
 = |K : F'| \cdot |F' : F|
 = 2 \cdot 2
 > 2 .$$
This contradicts the minimality of~$\aG$.

\section{Groups of type $F_4$} \label{F4Sect}

\begin{prop}
Let\/ $\aG$ be an absolutely almost simple $F$-group of type $F_4$, such that $\Frank \aG = 1$.
Then\/ $\aG$ contains an isotropic, simply connected, absolutely almost simply $F$-subgroup\/~$\aH$ of type~$C_3$, such that $\locrank \aH \ge 2$ for every $v \in \SG$.
\end{prop}

\begin{proof}
The Tits Classification \cite[p.~60]{[T66]} tells us that the Tits index of~$\aG$ is 
$$
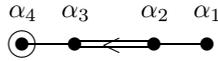
\begin{picture}(85,22)(-5,-10)
\put(00,00){\line(1,0){20}} \put(20,1.1){\line(1,0){30}}
\put(20,-1.2){\line(1,0){30}} \put(30,-3){$<$}
\put(50,00){\line(1,0){20}} 
\put(00,0){\circle*{5}}
\put(20,0){\circle*{5}} 
\put(50,0){\circle*{5}}
\put(70,0){\circle*{5}} 
\put(0,0){\circle{10}}
\put(-5,10){$\alpha_4$} \put(15,10){$\alpha_3$}
\put(45,10){$\alpha_2$} \put(65,10){$\alpha_1$}
\end{picture}
$$
(We number the simple roots as in \cite[p.~223]{[B68]}.)
Let $\aS$ be an $F$-split $1$-dimensional torus in
$\aG$ and  let $\aM$ be the corresponding semisimple $F$-anisotropic
kernel. 

For each $v \in \SG$, we have $\locrank \aG > 1$, so
the Tits Classification \cite[p.~60]{[T66]} implies $\aG$ is
split over~$F_v$. Hence, by Weak Approximation and the Hasse Principle, there is a quadratic extension~$L$ of~$F$, such that $L$ splits~$\aG$ and
	\begin{equation} \label{F4NotMin-LinFv}
	\text{$L \subset F_v$, for every $v \in \SG$.} 
	\end{equation}
Since $L$ splits~$\aG$ (and hence splits~$\aM$), there is an $L$-split maximal $F$-torus~$\aT$ of~$\aM$, such that the standard semisimple subgroup~$\aG_{\alpha_1}$ generated by the roots $\pm\alpha_1$ is defined over~$F$ \see{Malpha}.
Let $\aR = \aT \cap \aG_{\alpha_1} \subset \aM$, so $\aR$ is a $1$-dimensional, $L$-split, anisotropic $F$-torus.

Let $\aH$ be the semisimple part of the identity component of $\aC_{\aG}(\aR)$. 
We know that $\aH$ is defined over~$F$ (since $\aR$ is defined over~$F$).
It is easy to see that $\aH$ is the standard semisimple subgroup generated by the roots 
	$$ \pm \alpha_3, \pm \alpha_4, \pm (\alpha_1 + 2 \alpha_2 + 2 \alpha_3). $$ 
Thus, $\aH$ is of type $C_3$, so it is (absolutely) almost simple over~$F$. 
Also, $\aH$ is simply connected. (Note that $G$ is simply connected because it is of type~$F_4$, and see~\pref{WhichSC}.)
Furthermore, since $\aH$ has absolute rank~$3$, we have $\aC_{\aG}(\aR) = \aH \aR$.
\begin{itemize}
\item Since $\aR \subset \aM$, we know $\aS \subset \aC_{\aG}(\aR) = \aH \aR$. 
Since $\aS$ is isotropic over~$F$, and $\aR$ is anisotropic, this implies $\aH$ is isotropic over~$F$.
\item By construction, $L$ splits both~$\aG$ and~$\aR$; therefore, $\aC_{\aG}(\aR)$ contains an $L$-split maximal torus of~$\aG$. Since $\aC_{\aG}(\aR) = \aH \aR$, we conclude that $\aH$ splits over~$L$. From \pref{F4NotMin-LinFv}, we conclude that $\aH$ splits over~$F_v$, for every $v \in \SG$, so $\locrank \aH > 1$. \qedhere
\end{itemize}
\end{proof}

\begin{cor} \label{F4NotMin}
If\/ $\aG$ is of type~$F_4$, then $\aG$ is not minimal.
\end{cor}

\section{Groups of type $\out{3,6}D4$} \label{TrialitySect}

The following theorem may be of independent interest. The proof makes no use of our standing assumption that $F$ is an algebraic number field --- it suffices to assume only that $\mathop{\rm char} F \neq 2$.

\begin{thm} \label{TrialityHasSL2}
Let\/ $\aG$ be an absolutely almost simple $F$-group of type $\out3D4$ or~$\out6D4$, such that $\Frank \aG = 1$.
Then there exists an extension field~$K$ of~$F$, such that\/ $\res{K/F} \aSL_2$ is isogenous to an $F$-subgroup of\/~$\aG$, and\/ $|K:F| = 4$.
\end{thm}

\begin{proof}
We start with notation:
\begin{itemize}
\item Let $\aS$ be a maximal $F$-split torus of~$\aG$.
\item Let $\aM = [\aC_{\aG}(\aS),\aC_{\aG}(\aS)]$ be the semisimple $F$-anisotropic kernel of~$\aG$.
\item It is well known \cite[Thm.~43.8 and Prop.~43.9, p.~555]{BookInvols} that there exist 
	\begin{itemize}
	\item a cubic extension~$L$ of~$F$, 
	and
	\item a quaternion algebra $D = (a,b_1)_L$ over~$L$, 
	\end{itemize}
	such that
	\begin{itemize}
	\item $\aM$ is isogenous to $\res{L/F} \aSL_1(D)$,
	\item $a \in F$,
	\item $b_1 \in L$,
	and
	\item $N_{L/F} (b_1) = 1$.
	\end{itemize}
Because $\res{L/F} \aSL_1(D)$ is anisotropic, we know that $D$ is a division algebra.
\item Let $P = L \bigl[ \sqrt{b_1} \bigr]$, so $P$ is isomorphic to a maximal subfield of~$D$.
\item Let $\widetilde{P}$ be the Galois closure of $P$ over~$F$. 
\item There is a maximal $F$-torus $\aT$ of~$\aM$ that is isogenous to $\res{L/F} \left( \res{P/L}^{(1)} \Gmult \right)$.
\item Let 
	$$\Phi_H^+ = \{ \alpha_2, \alpha_2 + \alpha_1 + \alpha_3, \alpha_2 + \alpha_1 + \alpha_4, \alpha_2 + \alpha_3 + \alpha_4 \}$$
and
	$$ \Phi_H = \{\, \pm \alpha \mid \alpha \in \Phi_H^+ \,\} ,$$
where $\{\alpha_1,\alpha_2,\alpha_3,\alpha_4\}$ is a base of the roots of~$\aG$ with respect to the maximal torus $\aS \aT$, numbered as in Figure~\ref{TrialityTitsFig}.
\item Let $\aH$ be the standard subgroup of~$\aG$ generated by the roots in~$\Phi_H$.
Since the roots in $\Phi_H^+$ are pairwise orthogonal, it is obvious that $\aH$ is a semisimple group that is of type $A_1 \times A_1 \times A_1 \times A_1$ over the algebraic closure~$\overline{F}$.
\end{itemize}

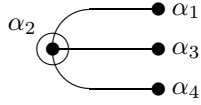
\begin{figure}[ht]
\begin{picture}(110,40)(-25,-15)
\put(00,00){\line(1,0){40}} \put(40,0){\oval(80,30)[l]}
\put(0,0){\circle*{5}} \put(40,0){\circle*{5}}
\put(40,15){\circle*{5}} \put(40,-15){\circle*{5}}
\put(00,0){\circle{12}}  
\put(-17,7){$\alpha_2$} \put(45,-2){$\alpha_3$}
\put(45,13){$\alpha_1$} \put(45,-17){$\alpha_4$}
\end{picture}
\caption{The Tits index of the trialitarian group~$\aG$.}
\label{TrialityTitsFig}
 \end{figure}

Since $\aM$ and $\aT$ are defined over~$F$, the Galois group $\Gal(\widetilde{P}/F)$ acts on the set
	$$\Phi_M = \{ {\pm \alpha_1}, \pm \alpha_3, \pm \alpha_4 \} $$
of roots of~$\aM$. Letting $b_3$ and~$b_4$ be the Galois conjugates of~$b_1$ (over~$F$), it is clear that $\Gal(\widetilde{P}/F)$ also acts on
	$$ B = \left\{ {\pm \sqrt{b_1}}, \pm \sqrt{b_3}, \pm \sqrt{b_4} \,\right\} .$$
It is easy to see that these two actions are isomorphic (because both are transitive and have $\Gal(\widetilde{P}/P)$ as the stabilizer of a point). Therefore, after renumbering and choosing the signs of the square roots appropriately, we know, for any $\varphi \in \Gal(\widetilde{P}/F)$, that there exist $\varepsilon_1, \varepsilon_3, \varepsilon_4 \in \{0,1\}$ and a permutation~$\sigma$ of $\{1,3,4\}$, such that
	$$ \text{$\varphi(\alpha_i) = (-1)^{\varepsilon_i} \alpha_{\sigma(i)}$ and $\varphi \bigl( \sqrt{b_i} \bigr) = (-1)^{\varepsilon_i} \sqrt{b_{\sigma(i)}}$ for $i = 1,3,4$.} $$ 
Since 
	$$\sqrt{b_1} \sqrt{b_3} \sqrt{b_4} = \pm \sqrt{b_1 b_3 b_4} = \pm \sqrt{N_{L/F}(b_1)} = \pm \sqrt{1} \in F ,$$
we know that $\sqrt{b_1} \sqrt{b_3} \sqrt{b_4}$ is fixed by~$\varphi$; therefore $\varepsilon_1 + \varepsilon_3 + \varepsilon_4$ is even. Hence,
	\begin{equation} \label{Epsilon0or2}
	\text{$\#\{\, i \mid \varepsilon_i \neq 0  \,\}$ is either $0$ or $2$.}
	\end{equation}
Let $\mu  = 2 \alpha_2 + \alpha_1 + \alpha_3 + \alpha_4$ be the maximal root of~$\aG$. The restriction of~$\mu$ to~$\aS$ is different from the restriction of any other root, so $\mu$ must be fixed by every element of $\Gal(\widetilde{P}/F)$. Therefore
	\begin{align*}
	2 \alpha_2 + \sum_{i \in \{1,3,4\}} \alpha_i
	&= \varphi \left( 2 \alpha_2 + \sum_{i \in \{1,3,4\}} \alpha_i \right) \\
	&= 2 \varphi(\alpha_2) + \sum_{i \in \{1,3,4\}} (-1)^{\varepsilon_i} \alpha_{\sigma(i)} 
	, \end{align*}
so
	\begin{equation} \label{TrialityPf-Phi(alpha2)}
	\varphi(\alpha_2) = \alpha_2 + \sum_{i \in \{1,3,4\}} \varepsilon_{\sigma^{-1}(i)} \alpha_i 
	. \end{equation}
From \pref{Epsilon0or2}, we conclude that $\varphi(\alpha_2) \in \Phi_H^+$.

Since $\varphi$ is an arbitrary element of $\Gal(\widetilde{P}/F)$, the conclusion of the preceding paragraph implies that $\Phi_H^+$ contains the entire orbit of~$\alpha_2$ under $\Gal(\widetilde{P}/F)$. In fact, it is easy to see that this orbit must be all of~$\Phi_H^+$. Since $\aT$, and hence~$\aG$, is obviously split over~$\widetilde{P}$, this implies that $\aH$ is defined over~$F$ and is almost $F$-simple. Also, since $\aS \subset \aH$, it is obvious that $\aH$ is isotropic over~$F$. Because $\aH$ is of type $A_1 \times A_1 \times A_1 \times A_1$ over the algebraic closure, it  is now clear that $\aH$ is isogenous to $\res{K/F} \aSL_2$, where $K$ is an extension of degree~$4$ over~$F$.
\end{proof}

\begin{rem}
The specific choice of the maximal subfield~$P$ of~$D$ is crucial in the above proof; it is important that $|\widetilde{P} : \widetilde{L}| = 4$, where $\widetilde{P}$ and $\widetilde{L}$ are the Galois closures of~$P$ and~$L$ over~$F$. If $P$ is chosen differently, then the action of the Galois group on the roots of~$\aG$ is different, and the standard subgroup generated by the roots in $\Phi_H$ is not defined over~$F$.
\end{rem}

\begin{rem}
Unfortunately, in the situation of Theorem~\ref{TrialityHasSL2}, it follows easily from Remark~\ref{WhichSC} that if $\aH$ is a subgroup of~$\aG$ that is isogenous to $\res{K/F} \aSL_2$, with $|K:F| = 4$, then $\aH$ is \emph{not} simply connected. Indeed, if $\aG$ is simply connected, then the fundamental group of~$\aH$ has order~$2$.
\end{rem}

\begin{cor}
If\/ $\aG$ is of type $\out3D4$ or $\out6D4$, then\/ $\aG$ is not minimal.
\end{cor}

\section{Groups of type $\out{1,2}E6$} \label{E6Sect}

We assume, in this section, that $\aG$ is of type~$E_6$. By Lemma~\ref{FrankMust1}, we may also assume $\Frank \aG = 1$. Then there are only two possibilities for~$\aG$ in the
Tits Classification \cite[pp.~58--59]{[T66]}:
 $$
\begin{array}{ccc}
 \out2E{6,1}^{29}& &\out2E{6,1}^{35} \\
\begin{picture}(70,50)(00,-20)
\put(05,02){\line(1,0){20}} \put(65,02){\oval(80,14)[l]}
\put(05,02){\circle*{5}} \put(25,02){\circle*{5}}
\put(65,-5){\circle*{5}} \put(65,9){\circle*{5}}
\put(45,9){\circle*{5}} \put(45,-5){\circle*{5}}
\put(65,02){\oval(10,25)} \put(-5,14){$\alpha_2$}
\put(20,14){$\alpha_4$} \put(40,14){$\alpha_3$}
\put(62,20){$\alpha_1$} \put(62,-20){$\alpha_6$}
\put(40,-16){$\alpha_5$}
\end{picture}
& \qquad &
\begin{picture}(65,50)(00,-20)
\put(05,02){\line(1,0){20}} \put(65,02){\oval(80,14)[l]}
\put(05,02){\circle*{5}} \put(25,02){\circle*{5}}
\put(65,-5){\circle*{5}} \put(65,9){\circle*{5}}
\put(45,9){\circle*{5}} \put(45,-5){\circle*{5}}
\put(05,02){\circle{10}} \put(-5,14){$\alpha_2$}
\put(20,14){$\alpha_4$} \put(40,14){$\alpha_3$}
\put(62,14){$\alpha_1$} \put(62,-16){$\alpha_6$}
\put(40,-16){$\alpha_5$}
\end{picture}
\end{array}
$$
(We number the simple roots as in \cite[p.~230]{[B68]}.)

 The two possible forms will be considered individually
 (in Theorems~\ref{E6(D4kernel)} and~\ref{E6(A5kernel)}).
 The proofs assume somewhat more background than those in previous sections.
 
\begin{thm} \label{E6(D4kernel)}
If\/ $\aG$ is a simply connected, absolutely almost simple $F$-group of type $\out2E{6,1}^{29}$, then\/ $\aG$ contains an isotropic, simply connected, absolutely almost simple $F$-subgroup\/~$\aH$ of type $\out2A5$, such that $\locrank \aH \ge 2$, for every archimedean place~$v$ of~$F$.
\end{thm}

Before proving this theorem, we recall the following result (and, for completeness, we provide a self-contained proof based on Galois cohomology). It does not require our standing assumption that $F$ is an algebraic number field.

\begin{lem}[{\cite[Rem.~2.10]{GaribaldiPetersson}}] \label{anis(E6)=spin}
If\/ $\aG$ is an absolutely almost simple $F$-group of type $\out2E{6,1}^{29}$, then the semisimple anisotropic kernel of\/~$\aG$ is isomorphic to\/ $\aSpin_8(f)$, for some quadratic form~$f$ on~$F^8$ with nontrivial discriminant.
\end{lem}

\begin{proof}
There is no harm in assuming $\aG$ is adjoint.
Let 
\begin{itemize}
	\item $K$ be the (unique) quadratic extension of~$F$ over which
$\aG$ becomes inner, 
	\item $\aG^q$ be a quasisplit, absolutely almost
simple, adjoint $F$-group of type $\out2E6$ that splits over~$K$ 
(so the Tits index of~$\aG^q$ is the diagram on the right in \pref{TitsIndex2E6overImag} below),
	\item $\aP$ be a minimal parabolic $F$-subgroup of~$\aG$,
	\item $\aP^q$ be a parabolic $F$-subgroup of~$\aG^q$ that is of the same type as~$\aP$ (so the semisimple part of~$\aP^q$ is the standard subgroup generated by the roots $\pm\alpha_2,\pm\alpha_3,\pm\alpha_4,\pm\alpha_5$),
	\item $\aR^q\aM^q$ be the the Levi subgroup of~$\aP^q$, where $\aR^q$ is its central torus and $\aM^q$ is its semisimple part,
	and
	\item $\xi\in Z^1(F,\aG^q)$, such that $\aG$ is (isomorphic to) the twisted group ${^{\xi}}\aG^q$.
\end{itemize}
\setcounter{step}{0}

\begin{step} \label{xiInM0}
The class of $\xi$ is in the image of the map $H^1(F,\aM^q)\to 
H^1(F,\aG^q)$.
\end{step}
It is well known that the image of the map
$H^1(F,\aP^q)\to H^1(F,\aG^q)$ consists of the classes of the
$1$-cocycles $\eta$ with the property that the twisted group $^{\eta}\aG^q$ has a
parabolic $F$-subgroup  of the same type as~$\aP^q$.
Thus, we may assume that $\xi\in Z^1(F,\aP^q)$. 
Then, since the unipotent radical of $\aP^q$ has trivial 
cohomology in dimension~$1$,
we may assume 
	$$ \xi\in Z^1(F,\aR^q \aM^q) .$$

Since the center of the universal cover of~$\aM^q$ has order~$4$, 
and the center of the universal cover of~$\aG^q$ has order~$3$, which is
relatively prime, we know that $\aM^q$ is simply connected. 
It is easy to check that $\aR^q$ is of the form
$\res{K/F}(\Gmult)$, and that $\aR^q\cap \aM^q$ is the entire center of~$\aM^q$,
which is isomorphic to $(\integer/2\integer)\times (\integer/2\integer)$, 
so
$\aR^q \cap \aM^q$ is precisely the $2$-torsion part ${_2\aR^q}$ of $\aR^q$; hence
	$$
	\frac{\aR^q \aM^q}{\aM^q} 
	\iso \frac{\aR^q}{\aR^q\cap \aM^q} 
	= \frac{\aR^q}{{_2\aR^q}} 
	\iso \aR^q
	\iso \res{K/F}(\Gmult)
	. $$
Therefore 
	$$ H^1 \left( F, \frac{\aR^q \aM^q}{\aM^q} \right) 
	\iso H^1\bigl( F,\res{K/F}(\Gmult) \bigr)  
	\iso H^1(K,\Gmult) 
	= 0 
	.$$
Since  $\xi\in Z^1(F,\aR^q \aM^q)$, the desired conclusion now follows from the exact sequence
	$$ H^1(F,\aM^q)
	\to H^1(F,\aR^q \aM^q) 
	\to H^1 \left( F, \frac{\aR^q \aM^q}{\aM^q} \right) 
	.$$

\begin{step}
Completion of the proof.
\end{step}
Recall that~$\aM^q$ is the standard subgroup of~$\aG^q$ generated by the roots 
$\pm\alpha_2,\ldots,\pm\alpha_5$. Thus, $\aM^q$ is a simply connected, 
quasisplit group of type 
$\out2D4$, so it is $F$-isomorphic to $\aSpin_8(f_0)$,
 where $f_0$ is a quasisplit quadratic form on~$F^8$. From Step~\ref{xiInM0},  
 we may assume that 
 	$$\xi\in Z^1(F,\aM^q) = Z^1 \bigl( F, \aSpin_8(f_0) \bigr)  .$$
Then the semisimple
$F$-anisotropic kernel~$\aM$ of~$\aG$ is the
twisted group 
	$${^\xi}\aM^q = {^\xi}{\aSpin_8(f_0)} \iso \aSpin_8(f) ,$$
for some quadratic form~$f$ on~$F^8$.
\end{proof}

\begin{proof}[{}{\bf Proof of Theorem~\ref{E6(D4kernel)}}]
Let 
\begin{itemize}
\item $\aS$ be a $1$-dimensional $F$-split torus in~$\aG$,
\item $\aM$ be the corresponding semisimple $F$-anisotropic kernel, so we may identify $\aM$ with $\aSpin(f)$, for some quadratic form~$f$ on~$F^8$ with nontrivial discriminant \see{anis(E6)=spin},
\item $K$ be the (unique) quadratic extension of~$F$ over which $\aG$ becomes a group of inner type,
\item $L$ be a totally imaginary quadratic extension of~$F$ (such that $L \neq K$),
\item $a \in F$, such that $L = F[\sqrt{a}]$, 
and
\item $\aR$ be the central torus in the reductive group $\aC_\aG(\aS)$, so $\aC_\aG(\aS) = \aR\aM$ is an almost-direct product, and $\aR$ is isogenous to $\res{K/F} \Gmult$.
\end{itemize}

Since $L \neq K$, we know that $\aG$ remains outer over~$L$. 
It is well known \cite[p.~385]{[PR]} that there are only two possibilities for the Tits index of a group of type $\out2E6$ over a totally imaginary number field:
 \begin{equation} \label{TitsIndex2E6overImag}
\begin{array}{ccc}
\begin{picture}(65,50)(00,-20)
\put(05,02){\line(1,0){20}} \put(65,02){\oval(80,14)[l]}
\put(05,02){\circle*{5}} \put(25,02){\circle*{5}}
\put(65,-5){\circle*{5}} \put(65,9){\circle*{5}}
\put(45,9){\circle*{5}} \put(45,-5){\circle*{5}}
\put(05,02){\circle{10}} \put(25,02){\circle{10}} 
\put(-5,14){$\alpha_2$}
\put(20,14){$\alpha_4$} \put(40,14){$\alpha_3$}
\put(62,14){$\alpha_1$} \put(62,-16){$\alpha_6$}
\put(40,-16){$\alpha_5$}
\end{picture}
& \qquad &
\begin{picture}(70,50)(00,-20)
\put(05,02){\line(1,0){20}} \put(65,02){\oval(80,14)[l]}
\put(05,02){\circle*{5}} \put(25,02){\circle*{5}}
\put(65,-5){\circle*{5}} \put(65,9){\circle*{5}}
\put(45,9){\circle*{5}} \put(45,-5){\circle*{5}}
\put(05,02){\circle{10}} \put(25,02){\circle{10}} 
\put(45,02){\oval(10,25)} \put(65,02){\oval(10,25)} 
\put(-5,14){$\alpha_2$}
\put(20,14){$\alpha_4$} \put(40,20){$\alpha_3$}
\put(62,20){$\alpha_1$} \put(62,-20){$\alpha_6$}
\put(40,-20){$\alpha_5$}
\end{picture}
\end{array}
\end{equation}
Furthermore, since the roots $\alpha_1$ and~$\alpha_6$ are circled in the Tits index
over~$F$, they must also be circled in the Tits index over~$L$.
Therefore, $\aG$ is quasisplit over~$L$; this means that $\aM$ is quasisplit over~$L$.
Hence, after a change of basis to diagonalize the form appropriately, we may write
	$$ f = \langle a_1, -a_1 a, a_2, -a_2 a, a_3, -a_3 a, b_1, b_2 \rangle .$$
Let $f' = \langle a_1, -a_1 a, a_2, -a_2 a \rangle$ be the restriction of~$f$ to the first $4$ coordinates.
By normalizing, we may assume $a_1 = 1$. Then $f'$ is the norm form of the quaternion algebra $D = (a,-a_2)_F$. (In other words, $f'$ is the 2-fold Pfister form $\langle 1,-a\rangle \otimes \langle 1,a_2 \rangle$.) Hence, 
	\begin{equation} \label{SO(f')=SL1DxSL1D}
	\aSpin_4(f') \iso \aSL_1(D) \times \aSL_1(D) 
	. \end{equation}
Let $\aM_1$ and $\aM_2$ be the two simple factors of~$\aSpin_4(f')$. 

Writing $f = f' \oplus f''$, we see that $\aM_1$ is normalized by $\aSpin_4(f') \cdot \aSpin_4(f'')$, which contains a maximal torus of~$\aM$. So $\aM_1$ is a standard subgroup. Since all roots of~$\aM$ are conjugate under the Weyl group, we may assume $\aM_1 = \aG_{\alpha_2}$ is the standard subgroup generated by the roots $\pm\alpha_2$.

Let $\aH$ be the identity component  of $\aC_\aG(\aM_1)$. 
Because $\aM_1$ is defined
over~$F$, we know that $\aH$ is defined over~$F$.
Furthermore, since $\aM_1 = \aG_{\alpha_2}$, it is easy to see that $\aH$ is the standard subgroup
of~$\aG$ generated by the roots 
	$$ \pm\alpha_3 , \pm\alpha_1,
	\pm(\alpha_2 + \alpha_3 + 2\alpha_4 + \alpha_5),
	\pm\alpha_6, \pm\alpha_5 .$$
Thus, $\aH$ is semisimple and simply connected \see{WhichSC}, and is of type $\out2A5$. Also, since $\aH$
contains the $F$-split torus $\aS$, we know that $\aH$ is $F$-isotropic. 

All that remains is to show, for every archimedean place~$v$ of~$F$, that $\locrank \aH \ge 2$.

\setcounter{case}{0}

\begin{case}
Assume $\aG$ is inner over~$F_v$.
\end{case}
This assumption implies $K \subset F_v$. Since $\aH$ contains the $2$-dimensional torus~$\aR$, which splits over~$K$ (because it is isogenous to~$\res{K/F} \Gmult$), we have $\locrank \aH \ge 2$.

\begin{case}
Assume $f'$ is isotropic over~$F_v$.
\end{case}
Since $\aSO_4(f')$ is isotropic, and its two simple factors $\aM_1$ and $\aM_2$ are isogenous \see{SO(f')=SL1DxSL1D}, we see that $\aM_2$ is isotropic. Since $\aM_2$ centralizes $\aM_1$, and is contained in~$\aM$, we see that $\aM_2$ is contained in the $F$-anisotropic kernel of~$\aH$. Thus, the $F$-anisotropic kernel of~$\aH$ is isotropic over~$F_v$; so $\locrank \aH \ge 2$.

\begin{case}
The remaining case.
\end{case}
Recall that $f = f' \oplus f''$. Since $f'$ is anisotropic over~$F_v$, we must have $F_v \iso \real$, and we may assume all of the coefficients of $f' = \langle a_1, -a_1 a, a_2, -a_2 a \rangle$ are positive in~$F_v$. 

Since $\aG$ is isotropic over~$F_v$ (indeed, it has been assumed to be isotropic over~$F$), we see, from the Tits Classification \cite[pp.~58--59]{[T66]} of real forms of~$E_6$, that $\locrank \aG \ge 2$. Hence, $\aM$ is isotropic over~$F_v$, so $f$ is isotropic over~$F_v$. Thus, some coefficient of $f''$ must be negative. On the other hand, because $\aG$ is outer over~$F_v$, we know that the discriminant of~$f$ is not~$1$, so the coefficients of~$f''$ cannot all be negative. Thus, $f''$ has both positive and negative coefficients, so $\aSpin_4(f'')$ is isotropic over~$F_v$. Since $\aSpin_4(f'')$ obviously centralizes $\aSpin_4(f') \supset \aM_1$, and is contained in~$\aM$, we see that $\aSpin_4(f'')$ is contained in the $F$-anisotropic kernel of~$\aH$. Thus, the $F$-anisotropic kernel of~$\aH$ is isotropic over~$F_v$; so $\locrank \aH \ge 2$.
\end{proof}

\begin{thm} \label{E6(A5kernel)}
If\/ $\aG$ is an absolutely almost simple $F$-group of type $\out2E{6,1}^{35}$, then\/ $\aG$ contains an isotropic, simply connected, absolutely almost simple $F$-subgroup~$\aH$ of type $\out3D4$ or $\out6D4$.
\end{thm}

\begin{proof}
We fix:
\begin{itemize}
\item a maximal $F$-split torus~$\aS$ of~$\aG$,
\item a maximal $F$-torus~$\aT$ that contains~$\aS$,
and
\item an ordering of the roots of~$\aG$ (with respect to the maximal torus~$\aT$). 
\end{itemize}
Let 
	$$ \mu = -(\alpha_1 + 2 \alpha_2 + 2 \alpha_3 + 3 \alpha_4 + 2 \alpha_5 + \alpha_6) , $$
be the minimal root of $\aG$, so $\aG$ has the following extended Tits index.
$$
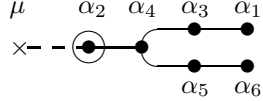
\begin{picture}(105,44)(-35,-17)
\put(-30,-2){$\times$}
\dashline{5,1}(-23,1)(0,1)
\put(00,01){\line(1,0){20}}
\put(40,01){\oval(40,14)[l]}
\put(40,8){\line(1,0){20}}
\put(40,-6){\line(1,0){20}}
\put(00,01){\circle*{5}}
\put(20,01){\circle*{5}}
\put(40,-6){\circle*{5}}
\put(40,8){\circle*{5}}
\put(60,8){\circle*{5}}
\put(60,-6){\circle*{5}}
\put(0,01){\circle{12}}
\put(-30,14){$\mu$}
\put(-4,14){$\alpha_2$}
\put(15,14){$\alpha_4$}
\put(35,14){$\alpha_3$}
\put(55,14){$\alpha_1$}
\put(55,-16){$\alpha_6$}
\put(35,-16){$\alpha_5$}
\end{picture}
$$
The standard subgroup $\aG_\mu$ of $\aG$
generated by the roots~$\pm\mu$ is isomorphic to
$\aSL_2$ over $F$, so $\aS \subset \aG_\mu$.

We may assume $G$ is simply connected (because the center of the universal cover of~$G$ has order~$3$, which is relatively prime to the order of the center of any group of type~$D_4$).
Let
\begin{itemize}
\item $K$ be the (unique) quadratic extension of~$F$ over
which $\aG$ becomes a group of inner type,
and
\item $\aM=[\aC_\aG(\aS),\aC_\aG(\aS)]$ the semisimple $F$-anisotropic kernel of~$\aG$, so~$\aM$ is generated by the roots
$\{\pm\alpha_1,\pm\alpha_3,\pm\alpha_4,\pm\alpha_5,\pm\alpha_6\}$.
\end{itemize}
Therefore $\aM$ is of type $\out2A5$ and becomes inner over~$K$, so, as is well known \cite[Prop.~2.18, p.~88]{[PR]}, we have
	$$ \mbox{$\aM$ is $F$-isomorphic to $\aSU_m(D,f,\tau)$,}$$
where $D$ is a central division algebra of index $d = 6/m$ over~$K$, with involution~$\tau$ of the second kind, such that $F$ is the fixed field of the restriction of~$\tau$ to~$K$, and $f$ is a nondegenerate Hermitian form on~$D^m$.

\begin{claim}
$D$ is a cubic division algebra over~$K$ {\rm(}and $m = 2${\rm)}.
\end{claim}
(This is known, but we provide a proof for completeness.)
We know that $\aG$ is a twisted form $\aG={^{\xi}\aG^q}$
of a quasisplit, almost simple, simply connected $F$-group $\aG^q$
of type~$\out2E6$ splitting over~$K$, where $\xi$ is a $1$-cocycle with
coefficients in the adjoint group~$\overline{\aG}^q$. 
In fact, there is a $1$-dimensional $F$-split torus~$\aS^q$ of~$\aG^q$, such that we may take $\xi$ to have its values in $\aC_{\overline{\aG}^q}(\overline{\aS}^q)$ (cf.\ \cite[Prop.~6.19, p.~339]{[PR]}). Write $\aC_{\aG^q}(\aS^q) = \aS^q \aM^q$, where $\aM^q$ is semisimple. Now $H^1 \bigl(F, \aC_{\overline{\aG}^q}(\overline{\aS}^q)/\overline{\aM}^q \bigr) = 0$ (because the coefficient group is an $F$-split torus), so we may take $\xi$ to have its values in~$\overline{\aM}^q$. Therefore $\aM =  {^{\xi}\aM^q}$.
 
Let $\aZ$ be the center of $\aG^q$ (note that $\aZ$ is contained in $\aM^q$), and let $\partial \colon H^1(F,\overline{\aM}^q)\to H^2(F,\aZ)$ be the connecting morphism. There is a cubic extension~$E$ of~$F$, such that the image of~$\partial \xi$ in $H^2(E,\aZ)$ is trivial \cite[Prop.~6.14, p.~334]{[PR]}. This means that the image of~$\xi$ in $H^1(E,\overline{\aM}^q)$ lifts to an element of $H^1(E,\aM^q)$, so $\aM$ is isomorphic over~$E$ to $\aSU_6(K \cdot E, f', \tau')$, where $\tau'$ is the Galois automorphism of $K \cdot E$ over~$E$, and $f'$~is a Hermitian form on $(K \cdot E)^6$. Therefore, $D \otimes_K (K \cdot E)$ is a matrix algebra. So $D$ is either~$K$ or a cubic division algebra over~$K$.

To complete the proof of the claim, we need only show $D \neq K$. Assume the contrary. Then $\tau$ is the Galois automorphism of~$K$ over~$F$, $f$~is a Hermitian form on~$K^6$, and $\aM \iso \aSU_6(K,f,\tau)$. For any archimedean place~$v$ of~$F$, the Tits Classification \cite[pp.~58--59]{[T66]} implies 
	$$\locrank \aG > 1 = \Frank \aG ,$$
so $\locrank \aM \ge 1$; thus, $f$ is $F_v$-isotropic. Then, since any Hermitian form in $\ge 3$ variables is isotropic at every nonarchimedean place, the Hasse Principle tells us that $f$~is $F$-isotropic. This contradicts the fact that $\aM$ is the $F$-anisotropic kernel. This completes the proof of the claim.

\bigbreak

Choose a basis $\{e_1,e_2\}$ of~$D^2$ that is orthogonal with respect to~$f$. By making a change of coordinates, we may assume $e_1 = (1,0)$ and $e_2 = (0,1)$. Then, letting
	$$ \text{$d_1=f(e_1,e_1)$ and $d_2=f(e_2,e_2)$,} $$
we have 
	$$f(x_1,x_2) = \tau(x_1) \, d_1 \, x_1 + \tau(x_2) \, d_2 \, x_2 .$$
Let $d= d_1^{-1} d_2\in D$. Then Lemma~\ref{NormalizeHerm} implies that we may assume $d_1 = 1$ and $d_2 = d$ (by replacing $\tau$ with $\int(d_1^{-1}) \circ \tau$). That is,
	$$ f = \langle 1, d \rangle .$$
For convenience, let us identify $\aM$ with $\aSU_2(D,f,\tau)$.

It is not difficult to see that there exists $x \in D$, such that $\tau(x) \, x \notin F$. Thus,
for a generic choice of the orthogonal basis $\{e_1,e_2\}$ (or merely multiplying $e_2$ by a generic element of~$D$), we have $d \notin F$. Since $\tau(d) = d$, this implies $d \notin K$. Therefore
\begin{itemize}
\item $E=K[d]$ is a maximal subfield in~$D$ (so it is cubic over~$K$),
\item $E$ is stable under~$\tau$,
and
\item $L= F[d]$ is a subfield of~$D$ that is cubic over~$F$. 
\end{itemize}
Consider the subgroup $\aM'=\aSU_2(E,f,\tau|_E)$ of~$\aM$. 
Writing $K = F[\sqrt{a}]$, for some $a \in F$, and 
letting $T$ be the quaternion algebra $T=(a,-d)_F$ over~$F$, we have
	$$\aM'\iso \res{L/F} \bigl( \aSL_1(T) \bigr) .$$

Let $\overline {K}$ be an algebraic closure of~$K$. Then $D \otimes_K \overline{K} \iso \Mat3(\overline{K})$, and the isomorphism may be taken so that $E \otimes_K \overline{K}$ maps to the diagonal matrices. Then the algebra $\Mat2(E)$, viewed as a subalgebra of $\Mat2(D) \otimes_K \overline{K} \iso \Mat6(\overline{K})$,
consists of matrices of the form 
$$
\begin{bmatrix}
\ast & 0 & 0 & \ast & 0 & 0\\
0 &\ast & 0 & 0 &\ast &0\\
0 & 0 & \ast & 0 & 0 & \ast\\
\ast & 0 & 0 & \ast & 0 & 0\\
0 &\ast & 0 & 0 &\ast & 0\\
0 & 0 & \ast & 0 & 0 & \ast
\end{bmatrix}
.$$
Hence, $\aM'$ is the standard subgroup generated by the roots $\pm\beta_1,\pm\beta_3,\pm\beta_4$, where
$$
\beta_1 = \alpha_1+\alpha_3+\alpha_4,\ 
\beta_3 = \alpha_3+\alpha_4+\alpha_5,
\ \beta_4 = \alpha_4+\alpha_5+\alpha_6.
$$
Let $\aH$ be the subgroup of $\aG$ generated by $\aM'$ and
$\aG_{\alpha_2}$. One easily checks that $\aH$ has type $D_4$,
contains $\aG_\mu$, and is simply connected \see{WhichSC}.

We now verify that $\aH$ is defined over~$F$.
Let $\sigma$ be a Galois automorphism of~$\overline{F}$ over~$F$.
Since $\aM'$ is defined over~$F$, we know that the set
$\{\pm\beta_1,\pm\beta_3,\pm\beta_4\}$ of roots of~$\aM'$ is invariant 
under~$\sigma$. 
Then, since 
	$$ -\mu = 2 \alpha_2 + \beta_1 + \beta_3 + \beta_4 $$
and $\mu$ is fixed by~$\sigma$ (because $\aG_\mu\simeq \aSL_2$),
the argument leading up to \eqref{TrialityPf-Phi(alpha2)} shows that 
$\sigma(\alpha_2)$ is a
root of~$\aH$. Thus, the set of roots of~$\aH$ 
is invariant under~$\sigma$.

Since $\aS \subset \aG_\mu \subset \aH$, we know $\aH$
is $F$-isotropic. Also, since $\aH$ contains $\aM'$, it is a
trialitarian group.
\end{proof}

\begin{cor} \label{E6(A5kernel)notmin}
If\/ $\aG$ is of type $E_6$, then\/ $\aG$ is not minimal.
\end{cor}

\begin{proof}
The conclusion is immediate from Theorem~\ref{E6(D4kernel)} if $\aG$ is of type $\out2E{6,1}^{29}$.

When $\aG$ is of type $\out2E{6,1}^{35}$, it suffices to observe that the subgroup $\aH$ provided
by Theorem~\ref{E6(A5kernel)} satisfies $\locrank \aH \ge 2$, for every archimedean
place~$v$ of~$F$. This follows from Theorem~\ref{TrialityHasSL2}, but it is also
easy to give a short direct proof. Note that, because $\aH$ is a trialitarian group of rank~$1$, its
Tits index is as shown in Figure~\ref{TrialityTitsFig}; thus, the root $\alpha_2$ is circled.
So $\alpha_2$ is also circled in the Tits index of~$\aH$ 
over~$F_v$. From the Tits Classification \cite[pp.~56--58]{[T66]}
of groups of type $D_4$ over~$\real$,
we see that this implies at least two roots are circled,
so $\locrank \aH \ge 2$.
\end{proof}

\begin{acks}
We thank an anonymous referee for detailed comments that clarified and shortened some of the proofs, and strengthened Theorem~\ref{E6(A5kernel)}.
We also thank Gopal Prasad for suggesting that our main theorem should be proved over general number fields. D.W.M.\ would like to thank Skip Garibaldi for helpful discussions.
\end{acks}

\end{document}